\documentclass[11pt]{article}

\usepackage[a4paper,margin=1in]{geometry}
\usepackage{amsmath,amssymb,amsthm,mathtools,bm}
\usepackage{enumitem}
\usepackage{hyperref}
\usepackage[nameinlink,capitalise]{cleveref}

\hypersetup{hidelinks}
\newtheorem{theorem}{Theorem}[section]
\newtheorem{lemma}[theorem]{Lemma}
\newtheorem{proposition}[theorem]{Proposition}
\newtheorem{corollary}[theorem]{Corollary}
\newtheorem{assumption}[theorem]{Assumption}
\newtheorem{definition}[theorem]{Definition}
\newtheorem{remark}[theorem]{Remark}

\crefname{assumption}{Assumption}{Assumptions}
\Crefname{assumption}{Assumption}{Assumptions}
\crefname{definition}{Definition}{Definitions}
\Crefname{definition}{Definition}{Definitions}
\crefname{remark}{Remark}{Remarks}
\Crefname{remark}{Remark}{Remarks}
\crefname{lemma}{Lemma}{Lemmas}
\Crefname{lemma}{Lemma}{Lemmas}
\crefname{theorem}{Theorem}{Theorems}
\Crefname{theorem}{Theorem}{Theorems}
\crefname{equation}{Eq.}{Eqs.}
\Crefname{equation}{Eq.}{Eqs.}

\newcommand{\R}{\mathbb R}
\newcommand{\Zp}{\mathbb Z_+}
\newcommand{\F}{\mathcal F}
\newcommand{\E}{\mathbb E}
\newcommand{\Uad}{\mathcal U_{ad}}
\newcommand{\ip}[2]{\left\langle #1,#2\right\rangle}
\newcommand{\norm}[1]{\left\|#1\right\|}

\title{Fully Coupled Nonlinear Forward-Backward Stochastic Difference Equations: Spectral Contraction and Infinite-Horizon Maximum Principle\thanks{This paper is supported by the National Key R\&D Program of China (No. 2022YFA1006101), the National Natural Science Foundation of China (No. 12371445), and the State Key Laboratory of Cryptography and Digital Economy Security, Shandong University (No. KFZD2505).}}
\author{Rui Chen\thanks{School of Mathematical Sciences, Fudan University, Shanghai 200433, China. Email: 22110180002@m.fudan.edu.cn}\ \ and\ \
  Qi  Zhang\thanks{Corresponding author. School of Mathematical Sciences and Laboratory of Mathematics for Nonlinear Science, Fudan University, Shanghai 200433, China; State Key Laboratory of Cryptography and Digital Economy Security, Shandong University, Jinan 250100, China. Email: qzh@fudan.edu.cn.}} 
\date{}

\begin{document}
\maketitle

\begin{abstract}
This paper develops an explicit spectral-contraction approach to studying fully coupled nonlinear
forward--backward stochastic difference equations on infinite
horizon and their applications to stochastic control. 
The estimates for this fully coupled system are assembled into an explicit two-dimensional nonnegative
matrix. Its spectral radius yields an explicit sufficient discount threshold,
and a corresponding equivalent weighted product norm is constructed to establish
the contraction property.
Moreover, for an infinite-horizon control problem with convex control constraints and an
accumulated discounted running cost, we derive a Pontryagin-type stochastic
maximum principle, its equivalent pointwise normal-cone formulation, and a
verification theorem. Finally, a recursive risk-adjusted portfolio example is given to demonstrate the applications of our theoretical results, and a projection-type sufficient optimality condition for this example is derived.
\end{abstract}



\section{Introduction}

Backward stochastic differential equations (BSDEs), introduced in their nonlinear form by Pardoux and Peng \cite{PardouxPeng1990}, have become a fundamental tool in stochastic analysis, mathematical finance, nonlinear expectations, and stochastic control. In a BSDE, the current value is determined recursively from a terminal condition and future information through a generator. The first component of the solution is commonly interpreted as a value, utility, or risk process, whereas the martingale integrand describes its sensitivity to the underlying noise. These interpretations lay a solid foundation for the important applications of recursive utility, nonlinear pricing, and dynamic risk evaluation; see Duffie and Epstein \cite{DuffieEpstein1992}, El Karoui, Peng and Quenez \cite{ElKarouiPengQuenez1997}, and Peng's theory of \(g\)-expectations \cite{Peng1997}.

The connection between backward stochastic equations and stochastic control is particularly evident in the stochastic maximum principle. Linear BSDEs arise as adjoint equations in the duality approach initiated by Bismut \cite{Bismut1973}. Nearly two decades later, Peng \cite{Peng1990} established a general stochastic maximum principle for nonlinear controlled diffusion systems. Maximum principles were subsequently extended to systems with more general noise structures, including jump processes; see Tang and Li \cite{TangLi1994}. In this framework, the forward state equation and the backward adjoint equation form a forward--backward stochastic system, while the Hamiltonian optimality condition provides a necessary condition for an optimal control.

The continuous-time theory of FBSDEs has developed some classical solvability methods. The four-step scheme of Ma, Protter and Yong \cite{MaProtterYong1994} deals with Markovian systems through associated partial differential equations. Monotonicity and continuation methods were developed by, among others, Hu and Peng \cite{HuPeng1995} and Yong \cite{Yong1997}. Fully coupled FBSDEs, where the forward equation is also coupled with the backward one, and their applications to optimal control are much more challenging. Relevant solvability results were established by Antonelli \cite{an}, Pardoux and Tang \cite{pa-ta}, and Peng and Wu \cite{PengWu1999}, to name but a few, and a comprehensive account can be found in Ma and Yong \cite{MaYong1999}. For infinite-horizon continuous-time FBSDEs, Peng and Shi \cite{PengShi2000} established solvability under monotonicity-type conditions.

In discrete time, backward stochastic difference equations (BS\(\Delta\)Es) provide the natural analogue of BSDEs when information arrives at discrete dates. Certainly, BSDEs and BS\(\Delta\)Es are totally different systems. For instance, the time intervals of discrete-time equations cannot tend to $0$, making classical analysis methods such as the It\^{o} formula no longer applicable. BS\(\Delta\)Es are useful not only as numerical analogues of continuous-time BSDEs, but also as intrinsic models for multi-period finance, dynamic risk measurement, recursive utility and control under sequential information updates. Cohen and Elliott \cite{CohenElliott2010} developed a finite-state BS\(\Delta\)E theory including existence, uniqueness, comparison and nonlinear expectation. Their subsequent work \cite{CohenElliott2011} considered more general discrete-time state spaces and nearly time-consistent nonlinear expectations. In mathematical finance, BS\(\Delta\)E techniques also support discrete-time conic finance and arbitrage-free bid--ask valuation; see Bielecki, Cialenco and Chen \cite{BieleckiCialencoChen2015}.

The infinite-horizon setting of BS\(\Delta\)Es presents additional difficulties compared with the finite-horizon case. The discussion of infinite horizon BS\(\Delta\)Es can be traced back to Allan and Cohen \cite{AllanCohen2016}, in which the ergodic BS\(\Delta\)Es (driven by finite-state Markov chains) with solutions defined on an arbitrarily large finite
time interval and its application to ergodic control are studied. 
Notice that on finite horizon, a BS\(\Delta\)E can typically be solved by backward induction from a prescribed terminal condition. But on infinite horizon, the absence of a prescribed terminal time prevents the direct use of backward induction, and instead the long-time behavior of the solution can be controlled in suitable exponentially weighted spaces. Weighted square-summability then provides the estimates needed for existence, uniqueness, and stability, and later ensures the vanishing of the boundary terms arising in the duality argument for the maximum principle. Based on this idea, Ji and Zhang \cite{JiZhang2024} developed a method for infinite-horizon BS\(\Delta\)E with solution defined in the whole infinite time
interval, and applied it to stochastic recursive control problems.

Fully coupled forward-backward stochastic difference equations (FBS\(\Delta\)Es) are more challenging to analyze. The forward state may depend on the backward recursive value and its martingale sensitivity, while the backward recursion depends on the future forward state. Fully coupled linear FBS\(\Delta\)Es and their applications to stochastic control were studied by Xu, Xie and Zhang \cite{XuXieZhang2017} and Xu, Zhang and Xie \cite{XuZhangXie2018}. 
Fully coupled nonlinear FBS\(\Delta\)Es have also received increasing attention. Niu, Meng, Li and Tang \cite{NiuMengLiTang2026} established finite-horizon solvability of nonlinear fully coupled FBS\(\Delta\)Es under domination--monotonicity conditions and the continuation method, with applications to LQ control.
For the infinite-horizon case, Ma, Li and Meng \cite{MaLiMeng2025} studied another class of fully coupled nonlinear FBS\(\Delta\)Es, in which the coefficients depend on backward variables defined through conditional expectations of future quantities, while the \(z\)-component is directly related to the future \(y\)-component through the martingale difference. Under domination--monotonicity conditions, \cite{MaLiMeng2025} established solvability by the continuation method and subsequently considered specific forward and backward LQ optimal control problems, but does not establish a stochastic maximum principle. 
For relevant control problems, Ji and Liu \cite{JiLiuFBSDeltaE} studied finite-horizon stochastic maximum principles for forward--backward stochastic difference systems, while Niu, Moon and Meng \cite{NiuMoonMeng2026} derived necessary and sufficient optimality conditions, via a Pontryagin-type maximum principle, for finite-horizon controlled fully coupled nonlinear FBS\(\Delta\)Es under a generalized monotonicity framework. 

The present paper develops an explicit spectral-contraction approach to
fully coupled nonlinear FBS\(\Delta\)Es on infinite horizon and applies
the resulting solvability theory to stochastic control. We consider systems
of the form
\begin{equation}\label{eq:fbsde}
	\left\{
	\begin{aligned}
		&X_{t+1}
		=X_t+b(t,X_t,Y_t,Z_t)
		+\sigma(t,X_t,Y_t,Z_t)\Delta W_t,\\
		&Y_t+Z_t\Delta W_t+\Delta N_t
		=e^{-\lambda}\bigl(
		Y_{t+1}
		+f(t+1,X_{t+1},Y_{t+1},Z_{t+1})
		\bigr),\\
		&X_0=x_0\in\mathbb{R}^m,
	\end{aligned}
	\right.
\end{equation}
where \(W\) is a square-integrable martingale with normalized conditional
covariance increments and \(N\) is a square-integrable martingale strongly
orthogonal to \(W\).
In contrast to approaches based
on domination--monotonicity and continuation arguments, our method uses
componentwise Lipschitz estimates and an explicit spectral contraction in
exponentially weighted spaces.

The backward equation is first constructed through finite-horizon approximation
combined with an a priori stability estimate. For the fully coupled system,
the solution is formulated in the mixed weighted class
\(
L^{2,\lambda}\times S^{2,\lambda}\times S^{2,\lambda},
\)
and the a priori estimates further yield the stronger
\(L^{2,\lambda}\)-summability of the backward components. The forward and
backward estimates are assembled into an explicit two-dimensional nonnegative
coupling matrix. Its spectral radius gives a computable sufficient discount
threshold, and a corresponding equivalent weighted product norm is constructed
to establish the contraction property. The same analysis also provides a
nonhomogeneous stability estimate, which is subsequently used to control state
perturbations and the variational equations arising in the control problem.

We then study a discrete-time infinite-horizon stochastic control problem governed by a
controlled fully coupled FBS\(\Delta\)E, with convex control constraints and an
accumulated discounted running cost satisfying quadratic-growth conditions.
The adjoint equation is itself a linear fully coupled FBS\(\Delta\)E whose
coefficients are formed from the first-order derivatives of the state
coefficients along an optimal trajectory. Although these derivatives enter the
state and adjoint systems in different forward--backward arrangements, there
exists a common sufficient discount threshold, above which both the controlled
state equation and the induced adjoint equation are well posed in the weighted
spaces.

Compared with the finite-horizon case, the infinite-horizon analysis requires
additional weighted estimates to justify the convergence of the perturbation
remainders and the limiting duality argument. In particular, a
main difficulty in deriving the stochastic maximum principle is to justify
the first-order approximation of the fully coupled nonlinear state system over
the whole infinite horizon.
To overcome this difficulty, we establish a controlled-convergence lemma
tailored to the discounted infinite-horizon setting. This lemma turns the
convergence of the perturbed trajectories into the convergence of the
nonlinear terms appearing in the first-order expansion, and yields the first-order approximation
of the state system and the corresponding expansion of the cost functional based on the
nonhomogeneous stability estimate.
This provides the key analytical foundation for the duality argument and the
resulting Pontryagin-type stochastic maximum principle. Under joint convexity
of the Hamiltonian in the state and control variables, we further prove a
verification theorem.

The rest of this paper is organized as follows. In \cref{sec:bsde}, we study infinite-horizon BS\(\Delta\)Es based on a mixed
weighted-space construction. In \cref{sec:fbsde}, we prove the solvability and a nonhomogeneous stability estimate of infinite-horizon fully coupled FBS\(\Delta\)Es. In \cref{sec:smp}, for the control problem defined by controlled fully coupled FBS\(\Delta\)Es with accumulated discounted cost, we derive the maximum principle and prove the verification theorem. In \cref{sec:example},  a recursive risk-adjusted portfolio example is demonstrated, for which the state and
adjoint discount thresholds are verified and an adjoint-based projection
characterization under convex portfolio constraints is given.

\section{Solvability of BS\texorpdfstring{\(\Delta\)}{Delta}E on Infinite Horizon}\label{sec:bsde}

The main result of this section is the solvability and an a priori stability estimate of infinite-horizon BS\(\Delta\)Es.  
Let \((\Omega,\F,\{\F_t\}_{t\in\Zp},\mathbb P)\) be a filtered probability space satisfying the usual completeness convention. Let \(W=\{W_t\}_{t\in\Zp}\) be an \(\R^d\)-valued square-integrable martingale. We write
\[
\Delta W_t:=W_{t+1}-W_t,
\]
and assume
\[
\E[\Delta W_t\mid\F_t]=0,\qquad
\E[\Delta W_t(\Delta W_t)^\top\mid\F_t]=I_d.
\]
No independent-increment assumption is imposed. The analysis below uses only the martingale-difference condition, the normalized conditional covariance, and the strong orthogonality of the residual martingales.

All processes are assumed to be adapted unless otherwise specified. Throughout the paper, every square-integrable martingale strongly orthogonal to \(W\) is normalized to start from zero. In particular,
\[
N_0=Q_0=R_0=0.
\]
Under this convention, uniqueness of a strongly orthogonal martingale component means uniqueness of the entire martingale process; without the normalization, only its increments would be determined.

\begin{definition}[Weighted spaces]\label{def:spaces}
For \(\lambda>0\), \(T\in\Zp\cup\{\infty\}\), and \(q\in\mathbb N\), let \(S^{2,\lambda}_{\F}(0,T;{\R}^q)\) be the space of \(\R^q\)-valued adapted processes \(X=\{X_t\}_{0\leq t\leq T}\) such that
\[
\norm{X}_{S^{2,\lambda}_{\F}(0,T;{\R}^q)}
:=
\sup_{0\leq t\leq T}
\left(\E[e^{-\lambda t}|X_t|^2]\right)^{1/2}<\infty.
\]
Let \(L^{2,\lambda}_{\F}(0,T;\R^q)\) be the space of \(\R^q\)-valued adapted processes \(X=\{X_t\}_{0\leq t\leq T}\) such that
\[
\norm{X}_{L^{2,\lambda}_{\F}(0,T;\R^q)}
:=
\left(\sum_{t=0}^{T}\E[e^{-\lambda t}|X_t|^2]\right)^{1/2}<\infty.
\]
Let $\mathcal M_{\perp,0}^{2,\lambda}(W;\R^q)$ be the space of \(\R^q\)-valued adapted processes \(X=\{X_t\}_{t\geq0}\) such that
$X\in S_{\F}^{2,\lambda}(0,\infty;\R^q)$, $X_0=0$, $\Delta X=\{\Delta X_t\}_{t\geq0}\in L_{\F}^{2,\lambda}(0,\infty;\R^q)$, and $X$ is strongly orthogonal to \(W\),
namely 
$\E[\Delta X_t\mid\F_t]=0$
and
$\E[\Delta X_t(\Delta W_t)^\top\mid\F_t]=0$,
$\mathbb{P}$-{\rm a.s.}  for any$t\geq0$.

Throughout the paper, $|\cdot|$ denotes the Frobenius norm
(which coincides with the Euclidean norm for vectors), while
$\|\cdot\|_{\mathrm{op}}$ denotes the induced operator norm for linear maps with respect to these norms.
\end{definition}

For brevity, whenever the state space is clear, we write
\[
\|X\|_\lambda^2:=\sum_{t=0}^{\infty}\E[e^{-\lambda t}|X_t|^2].
\]
The same convention is used by summing the corresponding weighted norms.
Notice on the infinite horizon,
\[
L^{2,\lambda}_{\F}(0,\infty;\R^q)\subset S^{2,\lambda}_{\F}(0,\infty;\R^q).
\]
These two spaces play different roles. Backward components are first considered in the natural pointwise weighted class \(S^{2,\lambda}\), which is sufficient for the pointwise conditional estimates and uniqueness arguments. The a priori estimates then show that these same components actually belong to \(L^{2,\lambda}\), which is the space used for infinite-horizon summability, stability estimates and controlled-convergence arguments.

We first consider the infinite-horizon BS\(\Delta\)E
\begin{equation}\label{eq:bsde}
Y_t+Z_t\Delta W_t+\Delta N_t
=e^{-\lambda}\left(Y_{t+1}+f(t+1,Y_{t+1},Z_{t+1})\right),
\qquad t\geq0,
\end{equation}
where for $t\geq0$, \(Y_t\in\R^n\), \(Z_t\in\R^{n\times d}\), and \(N\) is an \(\R^n\)-valued square-integrable martingale strongly orthogonal to \(W\).

\begin{assumption}[Measurable and Lipschitz condition]\label{ass:bsde-gen}
The function
\[
f:\Omega\times\Zp\times\R^n\times\R^{n\times d}\to\R^n
\]
is jointly measurable and adapted.
There exist nonnegative constants \(c_1,c_2\) such that, for any
\(t\in\Zp\) and \(\mathbb P\)-{\rm a.s.} \(\omega\in\Omega\),
the following inequality holds for all
\(y_1,y_2\in\mathbb{R}^n\) and
\(z_1,z_2\in\mathbb{R}^{n\times d}\):
\[
|f(t,y_1,z_1)-f(t,y_2,z_2)|
\leq
c_1|y_1-y_2|+c_2|z_1-z_2|.
\]
\end{assumption}

\begin{assumption}[Discount condition]\label{ass:bsde-discount}
The constants \(\lambda,c_1,c_2\) satisfy
\begin{equation*}\label{eq:bsde-discount}
e^\lambda>(1+c_1)^2+c_2^2.
\end{equation*}
\end{assumption}

\begin{assumption}[Integrability at the origin]\label{ass:bsde-origin}
\[
\sum_{t=0}^{\infty} \E\left[e^{-\lambda t}|f(t,0,0)|^2\right]<\infty.
\]
\end{assumption}
\begin{definition}[Natural weighted solution class]
\label{def:bsde-solution-class}
For \(\lambda >0\),
an adapted triple \((Y,Z,N)\) is called a solution to the
infinite-horizon BS\(\Delta\)E \eqref{eq:bsde}
(in the natural weighted solution class) if
\[
(Y,Z,N)\in
S_{\F}^{2,\lambda}(0,\infty;\R^n)
\times
S_{\F}^{2,\lambda}(0,\infty;\R^{n\times d})
\times
\mathcal M_{\perp,0}^{2,\lambda}(W;\R^n),
\]
and the equation holds \(\mathbb{P}\)-{\rm a.s.} for any \(t\geq0\).
\end{definition}

\begin{lemma}[A priori estimate]\label{lem:bsde-estimate}
Suppose  that \(g\) satisfies Assumptions~\ref{ass:bsde-gen} and \ref{ass:bsde-discount}. Let \(\phi,\phi'\in L^{2,\lambda}_{\F}(0,\infty;\R^n)\). If \((Y,Z,N)\) and \((Y',Z',N')\) are adapted solutions to
\[
Y_t+Z_t\Delta W_t+\Delta N_t
=e^{-\lambda}\left(Y_{t+1}+g(t+1,Y_{t+1},Z_{t+1})+\phi_{t+1}\right)
\]
and to the same equation with \(\phi\) replaced by \(\phi'\), respectively, then there exist constants \(\gamma\in(0,1)\) and \(C>0\), depending only on \(c_1,c_2,\lambda\), such that, for any \(t\ge0\),
\begin{align*}
&\E\left[e^{-\lambda t}\bigl(|\delta Y_t|^2+|\delta Z_t|^2+|\Delta\delta N_t|^2\bigr)\right] \\
&\qquad\leq
\gamma\E\left[e^{-\lambda(t+1)}\bigl(|\delta Y_{t+1}|^2+|\delta Z_{t+1}|^2\bigr)\right]
+C\E\left[e^{-\lambda(t+1)}|\delta\phi_{t+1}|^2\right],
\end{align*}
where \(\delta Y=Y-Y'\), \(\delta Z=Z-Z'\), \(\Delta\delta N=\Delta N-\Delta N'\), and \(\delta\phi=\phi-\phi'\).
\end{lemma}

\begin{proof}
Subtracting these two equations gives
\[
\delta Y_t+\delta Z_t\Delta W_t+\Delta\delta N_t
=e^{-\lambda}\left(\delta Y_{t+1}+\delta g_{t+1}+\delta\phi_{t+1}\right),
\]
where
\[
\delta g_{t+1}:=g(t+1,Y_{t+1},Z_{t+1})-g(t+1,Y'_{t+1},Z'_{t+1}).
\]
The left-hand side consists of an \(\mathcal F_t\)-measurable term, a stochastic integral with respect to \(\Delta W_t\), and an increment of a martingale strongly orthogonal to \(W\). Hence taking the conditional expectation with respect to \(\mathcal F_t\), we have
\[
\E\bigl[\langle \delta Y_t,\delta Z_t\Delta W_t\rangle\mid\mathcal F_t\bigr]=0,
\qquad
\E\bigl[\langle \delta Y_t,\Delta\delta N_t\rangle\mid\mathcal F_t\bigr]=0,
\]
\[
\E\bigl[\langle \delta Z_t\Delta W_t,\Delta\delta N_t\rangle\mid\mathcal F_t\bigr]=0,
\qquad
\E\bigl[|\delta Z_t\Delta W_t|^2\mid\mathcal F_t\bigr]=|\delta Z_t|^2.
\]
Consequently,
\begin{align}
\E\left[e^{-\lambda t}\bigl(|\delta Y_t|^2+|\delta Z_t|^2+|\Delta\delta N_t|^2\bigr)\right]=
 e^{-\lambda}\E\left[e^{-\lambda(t+1)}
\left|\delta Y_{t+1}+\delta g_{t+1}+\delta\phi_{t+1}\right|^2\right].
\label{eq:bsde-estimate-isometry-detailed}
\end{align}
Set
\[
a:=1+c_1,\qquad b:=c_2,
\qquad \rho_0:=a^2+b^2.
\]
Since \(g\) is Lipschitz,
\[
|\delta Y_{t+1}+\delta g_{t+1}|
\leq a|\delta Y_{t+1}|+b|\delta Z_{t+1}|.
\]
By the Cauchy--Schwarz inequality in \(\R^2\),
\begin{equation}\label{eq:bsde-rho0-bound}
|\delta Y_{t+1}+\delta g_{t+1}|^2
\leq
\rho_0\bigl(|\delta Y_{t+1}|^2+|\delta Z_{t+1}|^2\bigr).
\end{equation}
Due to Assumption~\ref{ass:bsde-discount}, \(e^{-\lambda}\rho_0<1\). Choose \(\varepsilon>0\) sufficiently small such that
\[
\gamma:=e^{-\lambda}(1+\varepsilon)\rho_0<1.
\]
Using
\[
|A+B|^2\leq (1+\varepsilon)|A|^2+\left(1+\frac1\varepsilon\right)|B|^2,
\]
with \(A=\delta Y_{t+1}+\delta g_{t+1}\) and \(B=\delta\phi_{t+1}\), and then applying \eqref{eq:bsde-rho0-bound}, we obtain
\[
\left|\delta Y_{t+1}+\delta g_{t+1}+\delta\phi_{t+1}\right|^2
\leq
(1+\varepsilon)\rho_0\bigl(|\delta Y_{t+1}|^2+|\delta Z_{t+1}|^2\bigr)
+C_\varepsilon |\delta\phi_{t+1}|^2.
\]
Combining the above inequality and \eqref{eq:bsde-estimate-isometry-detailed}, we have the desired estimate with \(C=e^{-\lambda}C_\varepsilon\).
\end{proof}

\begin{theorem}[Solvability of the infinite-horizon BS\(\Delta\)E]\label{thm:bsde-solvability}
Under Assumptions~\ref{ass:bsde-gen}--\ref{ass:bsde-origin}, \cref{eq:bsde} admits a unique solution.
Moreover, the solution satisfies the stronger summability property
\[
(Y,Z,\Delta N)\in
L^{2,\lambda}_{\F}(0,\infty;\R^n)
\times
L^{2,\lambda}_{\F}(0,\infty;\R^{n\times d})
\times
L^{2,\lambda}_{\F}(0,\infty;\R^n).
\]
\end{theorem}

\begin{proof}
Set
\[
g(t,y,z)=f(t,y,z)-f(t,0,0).
\]
Then \(g(t,0,0)=0\), and \(g\) has the same Lipschitz constants \(c_1,c_2\) as \(f\). For each \(k\in\mathbb N\), define
\[
\phi_t^k={\bf 1}_{\{1,\ldots,k+1\}}(t)f(t,0,0).
\]
For $k\in\mathbb{N}$, we first construct an approximation by a sequence of finite-horizon BS\(\Delta\)Es
\begin{equation}\label{cz1}
Y_t^k+Z_t^k\Delta W_t+\Delta N_t^k
=e^{-\lambda}\left(Y_{t+1}^k+g(t+1,Y_{t+1}^k,Z_{t+1}^k)+\phi_{t+1}^k\right),
\qquad 0\leq t\leq k,
\end{equation}
with terminal time $k+1$. Let \(Y_{k+1}^k=0\) and \(Z_{k+1}^k=0\). Suppose that we have known \((Y_{t+1}^k,Z_{t+1}^k)\), $0\leq t\leq k$. Define
\[
\Xi_{t+1}^k
:=e^{-\lambda}\left(Y_{t+1}^k+g(t+1,Y_{t+1}^k,Z_{t+1}^k)+\phi_{t+1}^k\right).
\]
Since \(\Xi_{t+1}^k\in L^2(\mathcal F_{t+1};\R^n)\), set
\[
Y_t^k:=\E[\Xi_{t+1}^k\mid\mathcal F_t],
\qquad
Z_t^k:=\E[\Xi_{t+1}^k(\Delta W_t)^\top\mid\mathcal F_t].
\]
The conditional Cauchy--Schwarz inequality and the covariance normalization give
\[
|Z_t^k|^2
\leq
\E[|\Xi_{t+1}^k|^2\mid\mathcal F_t]
\cdot\E[|\Delta W_t|^2\mid\mathcal F_t]
=d\,\E[|\Xi_{t+1}^k|^2\mid\mathcal F_t].
\]
By conditional Jensen's inequality, we also have
\[
\E|Y_t^k|^2
=\E\left|\E[\Xi_{t+1}^k\mid\mathcal F_t]\right|^2
\leq\E|\Xi_{t+1}^k|^2<\infty.
\]
Consequently, \(Y_t^k\in L^2(\mathcal F_t;\R^n)\) and \(Z_t^k\in L^2(\mathcal F_t;\R^{n\times d})\). Moreover,
\[
\E|Z_t^k\Delta W_t|^2
=
\E\!\left[\operatorname{tr}\!\left(
Z_t^k\E[\Delta W_t(\Delta W_t)^\top\mid\mathcal F_t](Z_t^k)^\top
\right)\right]
=
\E|Z_t^k|^2.
\]
It follows that
\[
\Delta N_t^k:=\Xi_{t+1}^k-Y_t^k-Z_t^k\Delta W_t
\]
is square-integrable. Because \(\E[\Delta W_t(\Delta W_t)^\top\mid\mathcal F_t]=I_d\), it satisfies
\[
\E[\Delta N_t^k\mid\mathcal F_t]=0,
\qquad
\E[\Delta N_t^k(\Delta W_t)^\top\mid\mathcal F_t]=0.
\]
Thus \(N^k_t:=\sum_{s=0}^{t-1}\Delta N_s^k\), with \(N^k_0=0\), is a square-integrable martingale strongly orthogonal to \(W\).
Moreover, this backward induction procedure also implies the uniqueness of solution to finite-horizon BS\(\Delta\)E \eqref{cz1} for any $k\in\mathbb{N}$.

Extend the finite-horizon solution to the whole half-line by setting
\(
\Delta N_{k+1}^k=0\), \(Y_t^k=Z_t^k=\Delta N_t^k=0\), \(t\ge k+2\).
Take $M,L\in\mathbb{N}$ and \(M>L\), and set
\[
A_t^{M,L}:=\E\left[e^{-\lambda t}\bigl(|Y_t^M-Y_t^L|^2+|Z_t^M-Z_t^L|^2+|\Delta N_t^M-\Delta N_t^L|^2\bigr)\right],
\]
\[
D_t^{M,L}:=\E\left[e^{-\lambda t}|\phi_t^M-\phi_t^L|^2\right].
\]
for any finite \(K\geq0\), by Lemma~\ref{lem:bsde-estimate} and the inequality
\(\E[e^{-\lambda t}(|\delta Y_t|^2+|\delta Z_t|^2)]\leq A_t^{M,L}\), we have
\[
\sum_{t=0}^{K}A_t^{M,L}
\leq
\gamma\sum_{t=1}^{K+1}A_t^{M,L}
+C\sum_{t=1}^{K+1}D_t^{M,L}.
\]
Consequently,
\[
(1-\gamma)\sum_{t=0}^{K}A_t^{M,L}
\leq
\gamma A_{K+1}^{M,L}+C\sum_{t=1}^{K+1}D_t^{M,L}.
\]
Choose \(K\geq M+1\). The extension of solution implies \(A_{K+1}^{M,L}=0\), and all nonzero terms have already been included in the finite sums. Hence
\begin{align*}
&\sum_{t=0}^{\infty}\E\left[e^{-\lambda t}\bigl(|Y_t^M-Y_t^L|^2+|Z_t^M-Z_t^L|^2+|\Delta N_t^M-\Delta N_t^L|^2\bigr)\right]\\
&\qquad\leq
\frac{C}{1-\gamma}\sum_{t=L+1}^{M+1}\E[e^{-\lambda t}|f(t,0,0)|^2].
\end{align*}
The last term tends to zero as \(L,M\to\infty\) by Assumption~\ref{ass:bsde-origin}. Hence \((Y^k,Z^k,\Delta N^k)\) is Cauchy in
\[
L^{2,\lambda}_{\mathcal F}(0,\infty;\R^n)
\times L^{2,\lambda}_{\mathcal F}(0,\infty;\R^{n\times d})
\times L^{2,\lambda}_{\mathcal F}(0,\infty;\R^n).
\]
Denote its limit by \((Y,Z,\Delta N)\), and set \(N_0=0\),
\[
N_t:=\sum_{s=0}^{t-1}\Delta N_s,
\qquad t\ge1.
\]
Note that the convergence in the weighted norm implies convergence in \(L^2\) for any fixed time \(t\). Hence
\[
f(t+1,Y_{t+1}^k,Z_{t+1}^k)\longrightarrow f(t+1,Y_{t+1},Z_{t+1})
\quad\text{in }L^2\ {\rm as}\ k\to\infty.
\]
due to the Lipschitz property of \(f\).
Taking the limit in the finite-horizon equation, we have BS\(\Delta\)E \eqref{eq:bsde} \(\mathbb P\)-a.s. for any \(t\ge0\). The martingale and strong orthogonality properties also follow from the continuity of conditional expectation. To see this, notice as $k\to\infty$,

\[\E\left|
\E[\Delta N_t^k-\Delta N_t\mid\mathcal F_t]
\right|^2
\leq
\E|\Delta N_t^k-\Delta N_t|^2
\longrightarrow0,
\]
so \(\E[\Delta N_t\mid\mathcal F_t]=0\). Moreover, as $k\to\infty$, we have
\[
\begin{aligned}
\E\big|
\E[\Delta N_t(\Delta W_t)^\top\mid\F_t]
\big|
&=
\E\big|
\E[(\Delta N_t-\Delta N_t^k)(\Delta W_t)^\top\mid\F_t]
\big|\\
&\le
\E\Big[
\E\big[
|(\Delta N_t-\Delta N_t^k)(\Delta W_t)^\top|
\mid\F_t
\big]
\Big]\\
&=
\E\big[
|(\Delta N_t-\Delta N_t^k)(\Delta W_t)^\top|
\big]\\
&=
\E\big[
|\Delta N_t-\Delta N_t^k|\,|\Delta W_t|
\big]\\
&\le
\big(\E|\Delta N_t-\Delta N_t^k|^2\big)^{1/2}
\big(\E|\Delta W_t|^2\big)^{1/2}
\longrightarrow0.
\end{aligned}
\]
Hence
\[
\E\!\left[\Delta N_t(\Delta W_t)^\top\mid\F_t\right]=0,
\qquad \mathbb P-{\rm a.s.}
\]
Moreover, since \(N\) is a square-integrable martingale with \(N_0=0\),
\[
\sup_{t\ge0}e^{-\lambda t}\E|N_t|^2
\leq
\sum_{s=0}^{\infty}e^{-\lambda s}\E|\Delta N_s|^2
=\|\Delta N\|_\lambda^2<\infty.
\]
Hence \(N\in S_{\F}^{2,\lambda}\), and consequently
\(N\in\mathcal M_{\perp,0}^{2,\lambda}(W;\R^n)\).
Thus \((Y,Z,N)\) is a solution. Moreover,
\[
(Y,Z,\Delta N)
\in
L^{2,\lambda}_{\F}\times L^{2,\lambda}_{\F}\times L^{2,\lambda}_{\F}.
\]


It remains to prove uniqueness in the natural \(S^{2,\lambda}\)-class. Let \((Y,Z,N)\) and \((Y',Z',N')\) be two solutions with \((Y,Z),(Y',Z')\in S^{2,\lambda}_{\F}\times S^{2,\lambda}_{\F}\). Lemma~\ref{lem:bsde-estimate} with \(\delta\phi=0\) gives
\[
A_t\leq \gamma B_{t+1},\ \ \ t\geq0,
\]
where
\[
A_t:=\E\left[e^{-\lambda t}\bigl(|\delta Y_t|^2+|\delta Z_t|^2+|\Delta\delta N_t|^2\bigr)\right],
\qquad
B_t:=\E\left[e^{-\lambda t}\bigl(|\delta Y_t|^2+|\delta Z_t|^2\bigr)\right].
\]
Since \(B_t\le A_t\) for $t\geq0$, by iteration we have for any \(k\ge1\),
\[
B_t\leq A_t\leq \gamma B_{t+1}\leq \gamma A_{t+1}\leq\cdots\leq \gamma^k B_{t+k}.
\]
The \(S^{2,\lambda}\)-assumption implies \(\sup_{s\ge0}B_s<\infty\). Taking \(k\to\infty\), we have \(B_t=0\) for any \(t\geq0\). The inequality \(A_t\le \gamma B_{t+1}\) then also implies \(A_t=0\) for any \(t\geq0\). Hence \(Y=Y'\), \(Z=Z'\), and \(\Delta N=\Delta N'\) for any time step. The proof is complete.
\end{proof}

\section{Solvability of Infinite-Horizon Fully Coupled FBS\texorpdfstring{\(\Delta\)}{Delta}E}\label{sec:fbsde}


In this section, we study fully coupled FBS\texorpdfstring{\(\Delta\)}{Delta}E \eqref{eq:fbsde},
in which \(X_t\in\R^m\), \(Y_t\in\R^n\), \(Z_t\in\R^{n\times d}\), \(b\) takes values in \(\R^m\), \(\sigma\) takes values in \(\R^{m\times d}\), and \(f\) takes values in \(\R^n\).

\begin{assumption}[Measurable and Lipschitz condition]\label{ass:fbsde-coeff}
The coefficients
\[
(b,\sigma,f):\Omega\times\Zp\times\R^m\times\R^n\times\R^{n\times d}\to\R^m\times\R^{m\times d}\times\R^n
\]
are jointly measurable and adapted.
There exist nonnegative constants
\[
L_b^x,L_b^y,L_b^z,\qquad
L_\sigma^x,L_\sigma^y,L_\sigma^z,\qquad
L_f^x,L_f^y,L_f^z
\]
such that, for any \(t\in\Zp\) and
\(\mathbb P\)-{\rm a.s.} \(\omega\in\Omega\), the following inequalities
hold for all \(x_1,x_2\in\mathbb{R}^m\),
\(y_1,y_2\in\mathbb{R}^n\), and
\(z_1,z_2\in\mathbb{R}^{n\times d}\):
\begin{align*}
&|b(t,x_1,y_1,z_1)-b(t,x_2,y_2,z_2)|\leq
L_b^x|x_1-x_2|+L_b^y|y_1-y_2|+L_b^z|z_1-z_2|,\\
&|\sigma(t,x_1,y_1,z_1)-\sigma(t,x_2,y_2,z_2)|\leq
L_\sigma^x|x_1-x_2|+L_\sigma^y|y_1-y_2|+L_\sigma^z|z_1-z_2|,\\
&|f(t,x_1,y_1,z_1)-f(t,x_2,y_2,z_2)|\leq
L_f^x|x_1-x_2|+L_f^y|y_1-y_2|+L_f^z|z_1-z_2|.
\end{align*}
\end{assumption}

\begin{assumption}[Discount condition]\label{ass:fbsde-discount}
Set
\begin{align*}
&B_{yz}:=\sqrt{(L_b^y)^2+(L_b^z)^2},
\qquad
\Sigma_{yz}:=\sqrt{(L_\sigma^y)^2+(L_\sigma^z)^2},\\
&\mathcal P_0:=(1+L_b^x)^2+(L_\sigma^x)^2,
\ \ \
\mathcal Q_0:=B_{yz}^2+\Sigma_{yz}^2,
\ \ \ \mathcal K:=(1+L_b^x)B_{yz}+L_\sigma^x\Sigma_{yz},\\
&F_x:=L_f^x,
\qquad
G:=\sqrt{(1+L_f^y)^2+(L_f^z)^2},
\end{align*}
and define
\begin{equation*}\label{eq:calligraphic-PQRS-def}
\mathcal P:=\mathcal P_0+\mathcal K,
\qquad
\mathcal Q:=\mathcal Q_0+\mathcal K,
\qquad
\mathcal R:=F_x^2+F_xG,
\qquad
\mathcal S:=G^2+F_xG.
\end{equation*}
We assume
\begin{equation}\label{eq:fbsde-new-discount}
e^\lambda>
\frac{\mathcal P+\mathcal S+\sqrt{(\mathcal P-\mathcal S)^2+4\mathcal Q\mathcal R}}{2}.
\end{equation}
\end{assumption}

\begin{assumption}[Integrability at the origin]\label{ass:fbsde-origin}
\[
\sum_{t=0}^{\infty}\E\left[e^{-\lambda t}\left(|b(t,0,0,0)|^2+|\sigma(t,0,0,0)|^2+|f(t,0,0,0)|^2\right)\right]<\infty.
\]
\end{assumption}

\begin{remark}\label{rem:discount-condition-interpretation}
The condition \eqref{eq:fbsde-new-discount} is an explicit sufficient condition obtained from the estimates below and the spectral radius of the associated coupling matrix. It may not be optimal. 
\end{remark}
\begin{definition}[Weighted solution class for the infinite-horizon
FBS\(\Delta\)E]
For \(\lambda>0\),
an adapted quadruple \((X,Y,Z,N)\) is called a solution to the
infinite-horizon FBS\(\Delta\)E \eqref{eq:fbsde} if
\[
(X,Y,Z,N)\in
L_{\mathcal F}^{2,\lambda}(0,\infty;\mathbb R^{m})
\times
S_{\mathcal F}^{2,\lambda}(0,\infty;\mathbb R^{n})
\times
S_{\mathcal F}^{2,\lambda}
(0,\infty;\mathbb R^{n\times d})
\times
\mathcal M_{\perp,0}^{2,\lambda}
(W;\mathbb R^{n}),
\]
and the equation holds \(\mathbb{P}\)-{\rm a.s.} for any \(t\geq0\).
\end{definition}

\begin{theorem}[Solvability of the infinite-horizon FBS\(\Delta\)E]\label{thm:fbsde-solvability}
Under Assumptions~\ref{ass:fbsde-coeff}--\ref{ass:fbsde-origin}, \cref{eq:fbsde} admits a unique solution, denoted by $(X,Y,Z,N)$. Moreover, the solution satisfies the stronger summability property
\[
(Y,Z,\Delta N)
\in
L^{2,\lambda}_{\F}(0,\infty;\R^n)
\times
L^{2,\lambda}_{\F}(0,\infty;\R^{n\times d})
\times
L^{2,\lambda}_{\F}(0,\infty;\R^n).
\]
\end{theorem}

\begin{proof}
We prove the result by a fixed point argument.

\emph{Step 1: Construction of the solution map.}
We construct the fixed point on the complete Hilbert product space
\[
L^{2,\lambda}_{\F}(0,\infty;\R^m)
\times
L^{2,\lambda}_{\F}(0,\infty;\R^n)
\times
L^{2,\lambda}_{\F}(0,\infty;\R^{n\times d}),
\]
which is continuously embedded in \(L^{2,\lambda}\times S^{2,\lambda}\times S^{2,\lambda}\). For $(\mathfrak x,\mathfrak y,\mathfrak z)\in L^{2,\lambda}_{\F}\times
L^{2,\lambda}_{\F}\times L^{2,\lambda}_{\F}$, define \((X,Y,Z)=\Phi(\mathfrak x,\mathfrak y,\mathfrak z)\) as follows:
\begin{equation}\label{eq:Phi-forward-detailed}
X_{t+1}=\mathfrak x_t+b(t,\mathfrak x_t,\mathfrak y_t,\mathfrak z_t)
+\sigma(t,\mathfrak x_t,\mathfrak y_t,\mathfrak z_t)\Delta W_t,
\qquad X_0=x_0,
\end{equation}
\begin{equation}\label{eq:Phi-backward-detailed}
Y_t+Z_t\Delta W_t+\Delta N_t
=e^{-\lambda}\left(Y_{t+1}+f(t+1,X_{t+1},Y_{t+1},Z_{t+1})\right).
\end{equation}
For the solution $X$ to the forward equation \eqref{eq:Phi-forward-detailed}, by the martingale isometry and the Lipschitz condition,
\begin{align*}
\E|X_{t+1}|^2
&=\E\Bigl[|\mathfrak x_t+b(t,\mathfrak x_t,\mathfrak y_t,\mathfrak z_t)|^2
+|\sigma(t,\mathfrak x_t,\mathfrak y_t,\mathfrak z_t)|^2\Bigr]\\
&\leq C\E\Bigl[|\mathfrak x_t|^2+|\mathfrak y_t|^2+|\mathfrak z_t|^2
+|b(t,0,0,0)|^2+|\sigma(t,0,0,0)|^2\Bigr].
\end{align*}
Multiplying by \(e^{-\lambda(t+1)}\) and summing over \(t\ge0\), by Assumption~\ref{ass:fbsde-origin} we have \(X\in L^{2,\lambda}_{\F}\).

For the solution \((Y,Z,N)\) to infinite-horizon BS\(\Delta\)E \eqref{eq:Phi-backward-detailed} with
fixed \(X\), we regard
\[
(y,z)\longmapsto f(t,X_t,y,z).
\]
as the generator, and it is easy to check that BS\(\Delta\)E \eqref{eq:Phi-backward-detailed} with
fixed \(X\) satisfies Assumptions~\ref{ass:bsde-gen}--\ref{ass:bsde-origin}.
By Theorem~\ref{thm:bsde-solvability}, we first have the uniqueness in \(S_{\mathcal{F}}^{2,\lambda}\times
S_{\mathcal F}^{2,\lambda}\times
\mathcal M_{\perp,0}^{2,\lambda}\) and then yields the stronger summability \((Y,Z,\Delta N)\in L^{2,\lambda}_{\F}\times L^{2,\lambda}_{\F}\times L^{2,\lambda}_{\F}\). Hence \(\Phi\) is a mapping from the Hilbert space $L^{2,\lambda}_{\F}\times L^{2,\lambda}_{\F}\times L^{2,\lambda}_{\F}$ into itself.

\emph{Step 2: Estimate for the forward part.}
Take two inputs \((\mathfrak x,\mathfrak y,\mathfrak z)\) and \((\mathfrak x',\mathfrak y',\mathfrak z')\), and denote their images by \((X,Y,Z,N)\) and \((X',Y',Z',N')\). Write
\[
\delta\mathfrak x_t=\mathfrak x_t-\mathfrak x'_t,
\quad
\delta\mathfrak y_t=\mathfrak y_t-\mathfrak y'_t,
\quad
\delta\mathfrak z_t=\mathfrak z_t-\mathfrak z'_t,
\]
and
\[
\delta X_t=X_t-X_t',\quad \delta Y_t=Y_t-Y_t',\quad \delta Z_t=Z_t-Z_t'.
\]
From \eqref{eq:Phi-forward-detailed},
\begin{equation*}\label{eq:delta-forward-detailed}
\delta X_{t+1}=\delta\mathfrak x_t+\delta b_t+\delta\sigma_t\Delta W_t,
\end{equation*}
where
\[
\delta b_t=b(t,\mathfrak x_t,\mathfrak y_t,\mathfrak z_t)
-b(t,\mathfrak x'_t,\mathfrak y'_t,\mathfrak z'_t),
\]
and \(\delta\sigma_t\) is defined similarly. Since \(\delta X_0=0\), the martingale isometry gives
\begin{equation}\label{eq:forward-isometry-detailed}
\E|\delta X_{t+1}|^2
=
\E\left[|\delta\mathfrak x_t+\delta b_t|^2+|\delta\sigma_t|^2\right].
\end{equation}
Let
\[
U_t:=|\delta\mathfrak x_t|,
\qquad
V_t:=\sqrt{|\delta\mathfrak y_t|^2+|\delta\mathfrak z_t|^2}.
\]
By Assumption~\ref{ass:fbsde-coeff} and the Cauchy--Schwarz inequality,
\[
|\delta b_t|\leq L_b^xU_t+B_{yz}V_t,
\qquad
|\delta\sigma_t|\leq L_\sigma^xU_t+\Sigma_{yz}V_t.
\]
Consequently,
\begin{align}
|\delta\mathfrak x_t+\delta b_t|^2+|\delta\sigma_t|^2
&\leq
\big((1+L_b^x)U_t+B_{yz}V_t\big)^2
+\big(L_\sigma^xU_t+\Sigma_{yz}V_t\big)^2 \notag\\
&=\mathcal P_0U_t^2+\mathcal Q_0V_t^2+2\mathcal K U_tV_t \notag\\
&\leq \mathcal P|\delta\mathfrak x_t|^2
+\mathcal Q\bigl(|\delta\mathfrak y_t|^2+|\delta\mathfrak z_t|^2\bigr).
\label{eq:forward-pointwise-detailed}
\end{align}
Multiplying \eqref{eq:forward-isometry-detailed} by \(e^{-\lambda(t+1)}\) and noticing \eqref{eq:forward-pointwise-detailed}, 
we obtain
\begin{equation}\label{eq:forward-estimate-sum-detailed}
\|\delta X\|_\lambda^2
\leq
 e^{-\lambda}\mathcal P\|\delta\mathfrak x\|_\lambda^2
+e^{-\lambda}\mathcal Q\left(\|\delta\mathfrak y\|_\lambda^2+\|\delta\mathfrak z\|_\lambda^2\right).
\end{equation}

\emph{Step 3: Estimate for the backward part.}
By \eqref{eq:Phi-backward-detailed}, $(\delta Y, \delta Z, \Delta\delta N)$ satisfies
\begin{equation*}\label{eq:delta-backward-detailed}
\delta Y_t+\delta Z_t\Delta W_t+\Delta\delta N_t
=e^{-\lambda}\left(\delta Y_{t+1}+\delta f_{t+1}\right),
\end{equation*}
where
\[
\delta f_{t+1}=f(t+1,X_{t+1},Y_{t+1},Z_{t+1})
-f(t+1,X'_{t+1},Y'_{t+1},Z'_{t+1}).
\]
A similar deduction as in Lemma~\ref{lem:bsde-estimate} yields
\begin{align}
\E\left[e^{-\lambda t}\left(|\delta Y_t|^2+|\delta Z_t|^2+|\Delta\delta N_t|^2\right)\right]=
e^{-\lambda}\E\left[e^{-\lambda(t+1)}|\delta Y_{t+1}+\delta f_{t+1}|^2\right].
\label{eq:backward-isometry-detailed}
\end{align}
By the Lipschitz condition for \(f\),
\begin{align}
|\delta Y_{t+1}+\delta f_{t+1}|^2
&\leq
F_x^2|\delta X_{t+1}|^2+G^2\bigl(|\delta Y_{t+1}|^2+|\delta Z_{t+1}|^2\bigr)+2F_xG|\delta X_{t+1}|\sqrt{|\delta Y_{t+1}|^2+|\delta Z_{t+1}|^2} \notag\\
&\leq
\mathcal R|\delta X_{t+1}|^2
+\mathcal S\bigl(|\delta Y_{t+1}|^2+|\delta Z_{t+1}|^2\bigr).
\label{eq:backward-pointwise-detailed}
\end{align}
Combining \eqref{eq:backward-isometry-detailed} and \eqref{eq:backward-pointwise-detailed}, we have 
\begin{equation*}\label{eq:backward-estimate-sum-detailed}
\|\delta Y\|_\lambda^2+\|\delta Z\|_\lambda^2
\leq
 e^{-\lambda}\mathcal R\|\delta X\|_\lambda^2
+e^{-\lambda}\mathcal S\left(\|\delta Y\|_\lambda^2+\|\delta Z\|_\lambda^2\right).
\end{equation*}
Notice that \eqref{eq:fbsde-new-discount} in Assumption \ref{ass:fbsde-discount} implies \(e^\lambda>\mathcal S\). Hence
\begin{equation}\label{eq:YZ-by-X-detailed}
\|\delta Y\|_\lambda^2+\|\delta Z\|_\lambda^2
\leq
\frac{e^{-\lambda}\mathcal R}{1-e^{-\lambda}\mathcal S}\|\delta X\|_\lambda^2.
\end{equation}

\emph{Step 4: Contraction and explicit construction of the weight.}
Set
\[
U_{\rm out}:=\|\delta X\|_\lambda^2,
\qquad
V_{\rm out}:=\|\delta Y\|_\lambda^2+\|\delta Z\|_\lambda^2,
\]
and
\[
U_{\rm in}:=\|\delta\mathfrak x\|_\lambda^2,
\qquad
V_{\rm in}:=\|\delta\mathfrak y\|_\lambda^2+\|\delta\mathfrak z\|_\lambda^2.
\]
Then \eqref{eq:forward-estimate-sum-detailed} and \eqref{eq:YZ-by-X-detailed} imply
\[
\binom{U_{\rm out}}{V_{\rm out}}
\leq
K_\lambda
\binom{U_{\rm in}}{V_{\rm in}},
\]
where
\[
K_\lambda=
\begin{pmatrix}
a_\lambda & b_\lambda\\[3pt]
c_\lambda & d_\lambda
\end{pmatrix}
:=
\begin{pmatrix}
e^{-\lambda}\mathcal P & e^{-\lambda}\mathcal Q\\[4pt]
\dfrac{e^{-2\lambda}\mathcal P\mathcal R}{1-e^{-\lambda}\mathcal S}
&
\dfrac{e^{-2\lambda}\mathcal Q\mathcal R}{1-e^{-\lambda}\mathcal S}
\end{pmatrix}.
\]
For the matrix \(K_\lambda\), its determinant is $0$. The eigenvalues of $K_\lambda$ are \(0\) and $e^{-\lambda}\mathcal P+\frac{e^{-2\lambda}\mathcal Q\mathcal R}{1-e^{-\lambda}\mathcal S}$, and the latter is the spectral radius of $K_\lambda$, denoted by $\rho(K_\lambda)$.
Let \(x=e^\lambda\). The inequality \(\rho(K_\lambda)<1\) is equivalent to
\[
\frac{\mathcal P}{x}+\frac{\mathcal Q\mathcal R}{x(x-\mathcal S)}<1.
\]
Since \eqref{eq:fbsde-new-discount} implies \(x>\mathcal S\), the above inequality is further equivalent to
\[
x^2-(\mathcal P+\mathcal S)x+\mathcal P\mathcal S-\mathcal Q\mathcal R>0.
\]
The larger root of this quadratic polynomial is
\[
\frac{\mathcal P+\mathcal S+\sqrt{(\mathcal P-\mathcal S)^2+4\mathcal Q\mathcal R}}{2}.
\]
Thus \eqref{eq:fbsde-new-discount} implies both \(e^{-\lambda}\mathcal S<1\) and \(\rho(K_\lambda)<1\).

Pick any number \(\kappa\in(\rho(K_\lambda),1)\). We now construct an equivalent norm in $L^{2,\lambda}_{\F}\times
L^{2,\lambda}_{\F}\times L^{2,\lambda}_{\F}$ by a constant \(m_*\). If \(c_\lambda>0\), define
\begin{equation*}\label{eq:explicit-m-weight}
m_*:=\frac{c_\lambda}{\kappa-a_\lambda}>0.
\end{equation*}
This is well defined since \(a_\lambda\leq \rho(K_\lambda)=a_\lambda+d_\lambda<\kappa\). Note that the determinant of \(K_\lambda\) is $0$, so \(a_\lambda d_\lambda=b_\lambda c_\lambda\). Then we have
\begin{equation}\label{cz2}
\left\{
\begin{aligned}
&m_*a_\lambda+c_\lambda\leq\kappa m_*,\\
&m_*b_\lambda+d_\lambda
=\frac{b_\lambda c_\lambda}{\kappa-a_\lambda}+d_\lambda
=\frac{a_\lambda d_\lambda}{\kappa-a_\lambda}+d_\lambda
=\frac{\kappa d_\lambda}{\kappa-a_\lambda}
\leq \kappa.
\end{aligned}
\right.
\end{equation}
If \(c_\lambda=0\), then \(\mathcal R=0\) and \(d_\lambda=0\), and define
\[
m_*:=
\begin{cases}
\dfrac{\kappa}{2b_\lambda},& b_\lambda>0,\\[6pt]
1,& b_\lambda=0.
\end{cases}
\]
Then \eqref{cz2} still holds.

For this \(m_*\), define
\[
\|(X,Y,Z)\|_{m_*,\lambda}^2
:=m_*\|X\|_\lambda^2+\|Y\|_\lambda^2+\|Z\|_\lambda^2.
\]
Obviously, \eqref{cz2} implies
\[
\|\Phi(\mathfrak x,\mathfrak y,\mathfrak z)-\Phi(\mathfrak x',\mathfrak y',\mathfrak z')\|_{m_*,\lambda}^2
\leq
\kappa\|(\mathfrak x,\mathfrak y,\mathfrak z)-(\mathfrak x',\mathfrak y',\mathfrak z')\|_{m_*,\lambda}^2.
\]
Thus \(\Phi\) is a strict contraction with contraction factor \(\sqrt\kappa<1\) in the complete Hilbert space $L^{2,\lambda}_{\F}\times
L^{2,\lambda}_{\F}\times L^{2,\lambda}_{\F}$ under the norm $\|(\cdot,\cdot,\cdot)\|_{m_*,\lambda}$. Banach's fixed point theorem yields a unique fixed point \((X,Y,Z)\in L^{2,\lambda}_{\F}\times
L^{2,\lambda}_{\F}\times L^{2,\lambda}_{\F}\). The corresponding  strongly orthogonal martingale \(N\in\mathcal M_{\perp,0}^{2,\lambda}\) is uniquely determined by the discrete-time martingale decomposition in the backward equation, and the fixed point solves \eqref{eq:fbsde}. 
Moreover, the embedding \(L^{2,\lambda}\subset S^{2,\lambda}\) gives a solution \[
(X,Y,Z,N)\in
L_{\mathcal F}^{2,\lambda}
\times
S_{\mathcal F}^{2,\lambda}
\times
S_{\mathcal F}^{2,\lambda}
\times
\mathcal M_{\perp,0}^{2,\lambda}.
\].

The uniqueness of the solution in $L_{\mathcal F}^{2,\lambda}
\times
L_{\mathcal F}^{2,\lambda}
\times
L_{\mathcal F}^{2,\lambda}
\times
\mathcal M_{\perp,0}^{2,\lambda}$ can be extended to the larger space $L_{\mathcal F}^{2,\lambda}
\times
S_{\mathcal F}^{2,\lambda}
\times
S_{\mathcal F}^{2,\lambda}
\times
\mathcal M_{\perp,0}^{2,\lambda}$. Let \((X,Y,Z,N)\in L_{\mathcal F}^{2,\lambda}
\times
S_{\mathcal F}^{2,\lambda}
\times
S_{\mathcal F}^{2,\lambda}
\times
\mathcal M_{\perp,0}^{2,\lambda}\) be a solution. For the backward equation, we write
\[
g(t,y,z):=f(t,0,y,z)-f(t,0,0,0),
\]
and
\[
\phi_t:=f(t,X_t,Y_t,Z_t)-f(t,0,Y_t,Z_t)+f(t,0,0,0).
\]
Then \(g(t,0,0)=0\), the generator \(g\) satisfies the BS\(\Delta\)E Lipschitz condition and discount condition in \((y,z)\), and
\[
|\phi_t|^2\le C\bigl(|X_t|^2+|f(t,0,0,0)|^2\bigr).
\]
Define
\[
d_s:=C\E\left[e^{-\lambda s}\bigl(|X_s|^2+|f(s,0,0,0)|^2\bigr)\right],
\qquad s\ge0,
\]
and \(B_t:=\E[e^{-\lambda t}(|Y_t|^2+|Z_t|^2)]\). By Lemma~\ref{lem:bsde-estimate}, 
we have for some \(0<\alpha<1\),
\[
B_t\le \alpha B_{t+1}+d_{t+1}.
\]
Iterating this inequality a finite number \(K\) of times yields
\[
B_t\le \sum_{j=1}^{K}\alpha^{j-1}d_{t+j}+\alpha^K B_{t+K},
\]
where the sequence \(\{B_t\}_{t\ge0}\) is bounded since \((Y,Z)\in S^{2,\lambda}\). 
Taking  \(K\to\infty\), we have
\[
B_t\le \sum_{j=1}^{\infty}\alpha^{j-1}d_{t+j}.
\]
For each finite \(T\), by Tonelli's theorem for nonnegative series, we have
\begin{align*}
\sum_{t=0}^{T}B_t
\le \sum_{s=1}^{\infty}d_s
   \sum_{t=0}^{\min\{T,s-1\}}\alpha^{s-t-1}\le \frac{1}{1-\alpha}\sum_{s=1}^{\infty}d_s.
\end{align*}
Taking \(T\to\infty\), by monotone convergence, we obtain
\[
\sum_{t=0}^{\infty}B_t
\le C\left(\|X\|_\lambda^2+\sum_{t=0}^{\infty}\E[e^{-\lambda t}|f(t,0,0,0)|^2]\right)<\infty.
\]
Thus \((Y,Z)\in L^{2,\lambda}_{\F}\times L^{2,\lambda}_{\F}\), and the equation itself yields \(\Delta N\in L^{2,\lambda}_{\F}\). Therefore \((X,Y,Z,N)\in L_{\mathcal F}^{2,\lambda}
\times
L_{\mathcal F}^{2,\lambda}
\times
L_{\mathcal F}^{2,\lambda}
\times
\mathcal M_{\perp,0}^{2,\lambda}\), in which the solution is unique. The proof is complete.
\end{proof}

Based on Theorem \ref{thm:fbsde-solvability}, we consider a nonhomogeneous system
\begin{equation}\label{eq:nonhomogeneous-fbsde}
\left\{
\begin{aligned}
&X_{t+1}=X_t+\mathcal B(t,X_t,Y_t,Z_t)+\varphi_t
+\bigl(\Sigma^0(t,X_t,Y_t,Z_t)+\psi_t\bigr)\Delta W_t,\\
&Y_t+Z_t\Delta W_t+\Delta N_t
=e^{-\lambda}\bigl(Y_{t+1}+\mathcal G(t+1,X_{t+1},Y_{t+1},Z_{t+1})+\theta_{t+1}\bigr),\\
&X_0=0,
\end{aligned}
\right.
\end{equation}
where
\[
\mathcal B:\Omega\times\Zp\times\R^m\times\R^n\times\R^{n\times d}\to\R^m,
\quad
\Sigma^0:\Omega\times\Zp\times\R^m\times\R^n\times\R^{n\times d}\to\R^{m\times d},
\]
and
\[
\mathcal G:\Omega\times\Zp\times\R^m\times\R^n\times\R^{n\times d}\to\R^n.
\]
\begin{corollary}[A nonhomogeneous stability estimate]\label{cor:nonhomogeneous-stability}
Suppose that $\mathcal B,\Sigma^0, \mathcal G$ satisfy Assumptions \ref{ass:fbsde-coeff} and \ref{ass:fbsde-discount} with $b=\mathcal B$, $\sigma=\Sigma^0$ and $f=\mathcal G$, respectively, and
$
\mathcal B_t(0,0,0)=0,
\Sigma^0_t(0,0,0)=0,
\mathcal G_t(0,0,0)=0$. If
$
\varphi\in L^{2,\lambda}_{\mathcal F}(0,\infty;\R^m),
\psi\in L^{2,\lambda}_{\mathcal F}(0,\infty;\R^{m\times d}),
\theta\in L^{2,\lambda}_{\mathcal F}(0,\infty;\R^n)$,
then \eqref{eq:nonhomogeneous-fbsde} admits a unique solution, denoted by $(X,Y,Z,N)$. 

Moreover,
\[
(Y,Z,\Delta N)
\in
L^{2,\lambda}_{\F}(0,\infty;\R^n)
\times
L^{2,\lambda}_{\F}(0,\infty;\R^{n\times d})
\times
L^{2,\lambda}_{\F}(0,\infty;\R^n),
\]
and there exists a constant \(C>0\), depending only on \(\lambda\) and the Lipschitz constants, such that
\begin{align}\label{eq:nonhomogeneous-stability-estimate}
\|X\|_\lambda^2+\|Y\|_\lambda^2+\|Z\|_\lambda^2+\|\Delta N\|_\lambda^2\leq
C\left(\|\varphi\|_\lambda^2+\|\psi\|_\lambda^2+\|\theta\|_\lambda^2\right).
\end{align}
If \(X_0\neq0\), the right-hand side is enlarged by a constant multiple of \(|X_0|^2\).
\end{corollary}

\begin{proof}
We only need to prove the estimate \eqref{eq:nonhomogeneous-stability-estimate}.

Set
\[
U:=\|X\|_\lambda^2,
\ \
V:=\|Y\|_\lambda^2+\|Z\|_\lambda^2,
\ \
D:=\|\Delta N\|_\lambda^2,\ \
A:=\|\varphi\|_\lambda^2+\|\psi\|_\lambda^2,
\ \
B:=\|\theta\|_\lambda^2.
\]
For the forward equation, put
\[
A_t^0:=X_t+\mathcal B_t(X_t,Y_t,Z_t),
\qquad
\Gamma_t^0:=\Sigma^0_t(X_t,Y_t,Z_t).
\]
As in the proof of Theorem~\ref{thm:fbsde-solvability}, the homogeneous Lipschitz bounds give
\begin{equation*}\label{eq:nonhom-forward-hom-bound}
|A_t^0|^2+|\Gamma_t^0|^2
\leq
\mathcal P |X_t|^2+\mathcal Q\bigl(|Y_t|^2+|Z_t|^2\bigr).
\end{equation*}
Using the martingale isometry and Young's inequality, for any \(\varepsilon>0\), we have
\begin{align*}
\E|X_{t+1}|^2
=\E\left[|A_t^0+\varphi_t|^2+|\Gamma_t^0+\psi_t|^2\right]\leq
(1+\varepsilon)\E\left[|A_t^0|^2+|\Gamma_t^0|^2\right]
+\left(1+\frac1\varepsilon\right)\E\left[|\varphi_t|^2+|\psi_t|^2\right].
\end{align*}
Multiplying by \(e^{-\lambda(t+1)}\), summing over \(t\ge0\), and using \(X_0=0\), we get
\begin{equation*}\label{eq:nonhom-forward-estimate}
U\leq
(1+\varepsilon)e^{-\lambda}(\mathcal P U+\mathcal Q V)+C_\varepsilon A.
\end{equation*}

For the backward equation, the isometry leads to
\begin{align*}
\E\left[e^{-\lambda t}\left(|Y_t|^2+|Z_t|^2+|\Delta N_t|^2\right)\right]=e^{-\lambda}\E\left[e^{-\lambda(t+1)}
\left|Y_{t+1}+\mathcal G_{t+1}(X_{t+1},Y_{t+1},Z_{t+1})+\theta_{t+1}\right|^2\right].
\end{align*}
The homogeneous part satisfies
\begin{equation*}\label{eq:nonhom-backward-hom-bound}
\left|Y_{t+1}+\mathcal G_{t+1}(X_{t+1},Y_{t+1},Z_{t+1})\right|^2
\leq
\mathcal R |X_{t+1}|^2+
\mathcal S\bigl(|Y_{t+1}|^2+|Z_{t+1}|^2\bigr).
\end{equation*}
Applying Young's inequality again and summing over \(t\ge0\) yields
\begin{equation}\label{eq:nonhom-backward-estimate}
V+D\leq
(1+\varepsilon)e^{-\lambda}(\mathcal R U+\mathcal S V)+C_\varepsilon B.
\end{equation}
In particular,
\begin{equation}\label{eq:nonhom-vector-estimate}
\binom{U}{V}
\leq
(1+\varepsilon)e^{-\lambda}
\begin{pmatrix}
\mathcal P&\mathcal Q\\
\mathcal R&\mathcal S
\end{pmatrix}
\binom{U}{V}
+C_\varepsilon
\binom{A}{B}.
\end{equation}
By \eqref{eq:fbsde-new-discount}, the spectral radius of the matrix
$
e^{-\lambda}
\begin{pmatrix}
\mathcal P&\mathcal Q\\
\mathcal R&\mathcal S
\end{pmatrix}
$
is strictly smaller than $1$. Hence we may choose sufficiently small \(\varepsilon>0\) such that the spectral radius of the matrix $
(1+\varepsilon)e^{-\lambda}
\begin{pmatrix}
\mathcal P&\mathcal Q\\
\mathcal R&\mathcal S
\end{pmatrix}
$ in \eqref{eq:nonhom-vector-estimate} is still strictly smaller than $1$. Since all the elements in the matrix  $
(1+\varepsilon)e^{-\lambda}
\begin{pmatrix}
\mathcal P&\mathcal Q\\
\mathcal R&\mathcal S
\end{pmatrix}
$ are nonnegative, \(I-(1+\varepsilon)e^{-\lambda}\begin{psmallmatrix}\mathcal P&\mathcal Q\\ \mathcal R&\mathcal S\end{psmallmatrix}\) has a nonnegative inverse given by the convergent Neumann series. Therefore there exists a constant \(C_0>0\) such that
\[
U+V\leq C_0(A+B).
\]
Moreover, by \eqref{eq:nonhom-backward-estimate} and the above inequality, there exist constants \(C_1, C>0\) such that
\[
D\leq V+D\leq C_1(U+V)+C_\varepsilon B\leq C_0C_1(A+B)+C_\varepsilon B\leq C(A+B).
\]
Then \eqref{eq:nonhomogeneous-stability-estimate} follows from the last two inequalities.
\end{proof}

\section{Stochastic Maximum Principle on Infinite Horizon}\label{sec:smp}


In this section, we study the infinite-horizon optimal control problem depicted by FBS\(\Delta\)E in our framework. To begin with, we give
a controlled FBS\(\Delta\)E
\begin{equation}\label{eq:controlled-fbsde}
\left\{
\begin{aligned}
&X_{t+1}=X_t+b(t,X_t,Y_t,Z_t,u_t)+\sigma(t,X_t,Y_t,Z_t,u_t)\Delta W_t,\\
&Y_t+Z_t\Delta W_t+\Delta N_t
=e^{-\lambda}\left(Y_{t+1}+f(t+1,X_{t+1},Y_{t+1},Z_{t+1},u_{t+1})\right),\\
&X_0=x_0.
\end{aligned}
\right.
\end{equation}
The cost functional is
\begin{equation*}\label{eq:cost}
J(u)=\sum_{t=0}^{\infty}\E\left[e^{-\lambda t}l(t,X_t,Y_t,Z_t,u_t)\right].
\end{equation*}
Let \(U\subset\R^r\) be a nonempty convex set. An admissible control is an \(U\)-valued, jointly measurable and adapted process \(u\) satisfying
\[
\sum_{t=0}^{\infty}\E[e^{-\lambda t}|u_t|^2]<\infty.
\]
The set of all admissible controls is denoted by \(\Uad\).
The problem is to find \(\bar u\in\Uad\) such that
\[
J(\bar u)=\inf_{u\in\Uad}J(u).
\]


\begin{assumption}[Lipschitz conditions and differentiability conditions]\label{ass:control-state}
The coefficients
\[
(b,\sigma,f):\Omega\times\Zp\times\R^m\times\R^n\times\R^{n\times d}\times U\to\R^m\times\R^{m\times d}\times\R^n
\]
are jointly measurable and adapted.  
There exist nonnegative constants
\[
L_b^x,L_b^y,L_b^z,L_b^u,
\qquad
L_\sigma^x,L_\sigma^y,L_\sigma^z,L_\sigma^u,
\qquad
L_f^x,L_f^y,L_f^z,L_f^u
\]
such that, for any \(t\in\Zp\) and
\(\mathbb P\)-{\rm a.s.} \(\omega\in\Omega\), the following inequalities
hold for all \(x_1,x_2\in\mathbb{R}^m\),
\(y_1,y_2\in\mathbb{R}^n\),
\(z_1,z_2\in\mathbb{R}^{n\times d}\), and \(u_1,u_2\in U\):
\begin{align*}
&|b(t,x_1,y_1,z_1,u_1)-b(t,x_2,y_2,z_2,u_2)|\leq
L_b^x|x_1-x_2|+L_b^y|y_1-y_2|+L_b^z|z_1-z_2|+L_b^u|u_1-u_2|,\\
&|\sigma(t,x_1,y_1,z_1,u_1)-\sigma(t,x_2,y_2,z_2,u_2)|\leq
L_\sigma^x|x_1-x_2|+L_\sigma^y|y_1-y_2|+L_\sigma^z|z_1-z_2|+L_\sigma^u|u_1-u_2|,\\
&|f(t,x_1,y_1,z_1,u_1)-f(t,x_2,y_2,z_2,u_2)|\leq
L_f^x|x_1-x_2|+L_f^y|y_1-y_2|+L_f^z|z_1-z_2|+L_f^u|u_1-u_2|.
\end{align*}
Moreover, for any \(t\in\Zp\) and
\(\mathbb P\)-{\rm a.s.} \(\omega\in\Omega\), the coefficients
\(b\), \(\sigma\), and \(f\) are continuously differentiable
with respect to \((x,y,z,u)\).
\end{assumption}

\begin{assumption}[Discount condition]\label{ass:control-discount}
Set 
\begin{align*}
&B_c:=\sqrt{(L_b^y)^2+(L_b^z)^2+(L_\sigma^y)^2},
\qquad
\Sigma_c:=\sqrt{(L_b^z)^2+(L_\sigma^y)^2+(L_\sigma^z)^2},\\
&\mathcal P_{0,c}:=(1+L_b^x+L_f^y)^2+(L_\sigma^x)^2+(L_f^z)^2,
\qquad
\mathcal Q_{0,c}:=B_c^2+\Sigma_c^2,\qquad F_c:=L_f^x
\\
& \mathcal K_c:=(1+L_b^x+L_f^y)B_c+(L_\sigma^x+L_f^z)\Sigma_c,
\qquad
G_c:=\sqrt{(1+L_b^x+L_f^y)^2+(L_\sigma^x)^2+(L_f^z)^2},
\end{align*}
and define
\[
\mathcal P_c:=\mathcal P_{0,c}+\mathcal K_c,
\qquad
\mathcal Q_c:=\mathcal Q_{0,c}+\mathcal K_c,
\qquad
\mathcal R_c:=F_c^2+F_cG_c,
\qquad
\mathcal S_c:=G_c^2+F_cG_c.
\]
We assume
\begin{equation}\label{eq:control-global-discount}
e^\lambda>\frac{\mathcal P_c+\mathcal S_c+\sqrt{(\mathcal P_c-\mathcal S_c)^2+4\mathcal Q_c\mathcal R_c}}{2}.
\end{equation}
\end{assumption}

\begin{assumption}[Integrability at the origin]\label{ass:control-origin}
There exists \(u^0\in U\) such that
\[
\sum_{t=0}^{\infty}\E\left[e^{-\lambda t}
\left(|b(t,0,0,0,u^0)|^2+|\sigma(t,0,0,0,u^0)|^2+|f(t,0,0,0,u^0)|^2\right)\right]<\infty.
\]
If \(0\in U\), we take \(u^0=0\).
\end{assumption}

\begin{assumption}[Cost functional conditions]\label{ass:control-diff-cost}
The running cost
\[
l:\Omega\times\mathbb Z_+\times\R^m\times\R^n\times\R^{n\times d}\times U\to\R
\]
is jointly measurable and adapted.
for any \(t\in\Zp\) and
\(\mathbb P\)-{\rm a.s.} \(\omega\in\Omega\),
it is continuously differentiable with respect to \((x,y,z,u)\), and
there exist a nonnegative constant
\(C\) and a nonnegative adapted process \(\varrho=\{\varrho_t\}_{t\geq0}\) with
\[
\sum_{t=0}^{\infty}\E[e^{-\lambda t}|\varrho_t|^2]<\infty
\]
such that, for any \(t\in\Zp\) and
\(\mathbb P\)-{\rm a.s.} \(\omega\in\Omega\), the following inequalities
hold for all \(x\in\mathbb{R}^m\), \(y\in\mathbb{R}^n\),
\(z\in\mathbb{R}^{n\times d}\), and \(u\in U\):
\[
|l(t,x,y,z,u)|\leq \varrho_t+C(1+|x|^2+|y|^2+|z|^2+|u|^2),
\]
and
\[
|l_x(t,x,y,z,u)|+|l_y(t,x,y,z,u)|+|l_z(t,x,y,z,u)|+|l_u(t,x,y,z,u)|
\leq \varrho_t+C(1+|x|+|y|+|z|+|u|).
\]
\end{assumption}

\begin{remark}\label{rem:any-control-origin}
Assumption~\ref{ass:control-origin} implies the corresponding origin integrability for any admissible control \(u\in\Uad\). Indeed, by the control Lipschitz condition,
\[
\sum_{t=0}^{\infty}\E\left[e^{-\lambda t}|b(t,0,0,0,u_t)|^2\right]
\leq
2\sum_{t=0}^{\infty}\E\left[e^{-\lambda t}|b(t,0,0,0,u^0)|^2\right]+2\sum_{t=0}^{\infty}\E\left[e^{-\lambda t}(L_b^u)^2|u_t-u^0|^2\right]<\infty,
\]
and the same estimate holds for \(\sigma\) and \(f\). Hence the controlled state equation is well defined for all \(u\in\Uad\), and thus by Theorem~\ref{thm:fbsde-solvability}, \(X\in L^{2,\lambda}_{\F}\) and \((Y,Z)\in S^{2,\lambda}_{\F}\times S^{2,\lambda}_{\F}\), and the same a priori estimate yields \((Y,Z,\Delta N)\in L^{2,\lambda}_{\F}\times L^{2,\lambda}_{\F}\times L^{2,\lambda}_{\F}\).
\end{remark}

\begin{lemma}[Absolute convergence of the cost functional]\label{lem:cost-well-defined}
Under Assumptions~\ref{ass:control-state}--\ref{ass:control-diff-cost}, the series defining \(J(u)\) is absolutely convergent for any \(u\in\Uad\). In particular, \(J(u)\in\R\).
\end{lemma}

\begin{proof}
For \(u\in\Uad\), by Theorem \ref{thm:fbsde-solvability} and Remark \ref{rem:any-control-origin}, the state equation \eqref{eq:controlled-fbsde} admits a unique solution \(X,Y,Z\in L_{\mathcal F}^{2,\lambda}\). 
Hence Assumption~\ref{ass:control-diff-cost} yields
\begin{align*}
&\sum_{t=0}^{\infty}\E\!\left[e^{-\lambda t}
|l(t,X_t,Y_t,Z_t,u_t)|\right]\\
&\quad\leq
\left(\sum_{t=0}^{\infty}e^{-\lambda t}\right)^{1/2}
\left(\sum_{t=0}^{\infty}\E[e^{-\lambda t}\varrho_t^2]\right)^{1/2}
+C\sum_{t=0}^{\infty}e^{-\lambda t}
+C\bigl(\|X\|_\lambda^2+\|Y\|_\lambda^2
+\|Z\|_\lambda^2+\|u\|_\lambda^2\bigr)<\infty.
\end{align*}
\end{proof}

Define a finite measure \(\nu\) in \(\Omega\times\mathbb Z_+\) by
\[
\nu(d\omega,t)
:=e^{-\lambda t}\mathbb P(d\omega),
\qquad \omega\in\Omega,\ t\ge0.
\]
Equivalently, for any jointly measurable, adapted and nonnegative process \(\Phi\),
\[
\int_{\Omega\times\mathbb Z_+}\Phi(\omega,t)\,d\nu
=
\sum_{t=0}^{\infty}
\E\left[e^{-\lambda t}\Phi(\omega,t)\right].
\]
Since \(\lambda>0\),
\[
\nu(\Omega\times\mathbb Z_+)
=
\sum_{t=0}^{\infty}e^{-\lambda t}
<\infty.
\]
In particular, for $R\in L_{\mathcal F}^{2,\lambda}$,
\[
\|R\|_{L_{\mathcal F}^{2,\lambda}}^2
=
\int_{\Omega\times\mathbb Z_+}|R_t(\omega)|^2\,d\nu
=
\|R\|_{L^2(\nu)}^2.
\]

\begin{lemma}[Controlled convergence on the discounted product space]\label{lem:controlled-convergence}
\emph{(i)}
Suppose that \(\{A^\varepsilon\}_{0<\varepsilon\leq1}\) is a family
of adapted operator-valued processes satisfying
\(
\sup_{0<\varepsilon\le1}
\operatorname*{ess\,sup}_{(\omega,t)}
\|A_t^\varepsilon(\omega)\|_{\mathrm{op}}
\\<\infty.
\)
Suppose further that
\(
\|A^\varepsilon\|_\lambda\longrightarrow0
\) as \(\varepsilon\to0\).
Then, for any \(R\in L_{\mathcal F}^{2,\lambda}\),
\(
\|A^\varepsilon R\|_\lambda\longrightarrow0\) as \(\varepsilon\to0\).
Moreover, if
\(
\|R^\varepsilon-R\|_\lambda\longrightarrow0
\) as \(\varepsilon\to0\),
then
\(
\|A^\varepsilon R^\varepsilon\|_\lambda\longrightarrow0
\) as \(\varepsilon\to0
\).


\emph{(ii)} Denote by \(\mathcal X\) and \(\mathcal Y\) two finite-dimensional normed spaces.
Suppose that $\{S^\varepsilon\}_{0<\varepsilon\leq1}\subset L_{\mathcal F}^{2,\lambda}(0,\infty;\mathcal X)$ and
\(S^\varepsilon\to S\) in \(L_{\mathcal F}^{2,\lambda}(0,\infty;\mathcal X)\) as $\varepsilon\to0$, and \(a=\{a_t\}_{t\ge0}\in L_{\mathcal F}^{2,\lambda}(0,\infty;\R^+)\).
for any \(t\ge0\), let
\[
(\omega,s)\longmapsto\Phi(\omega,t,s),\qquad \omega\in\Omega,\ s\in \mathcal X,
\]
be a jointly measurable, adapted and a.s. continuous in $s$ mapping from $\Omega\times \mathcal X$ to \(\mathcal Y\). If
\[
|\Phi(t,s)|_{\mathcal Y}\le a_t+C(1+|s|_{\mathcal X}),
\]
Then
\[
\|\bar\Phi^\varepsilon-\Phi(\cdot,S)\|_{\lambda}\longrightarrow0\ \ \ {\rm as}\ \varepsilon\to0,
\]
where
\[
\bar\Phi^\varepsilon(t):=\int_0^1
\Phi\bigl(t,S_t+\rho(S_t^\varepsilon-S_t)\bigr)\,d\rho.
\]

\end{lemma}

\begin{proof}

For (i), let
\[
M:=
\sup_{0<\varepsilon\le1}
\operatorname*{ess\,sup}_{(\omega,t)}
\|A_t^\varepsilon(\omega)\|_{\mathrm{op}}
<\infty.
\]
Fix \(R\) with \(\|R\|_\lambda<\infty\). for any \(K>0\),
\begin{align*}
\|A^\varepsilon R\|_\lambda^2
&=
\int_{\{|R|\le K\}}
|A^\varepsilon R|^2\,d\nu
+
\int_{\{|R|>K\}}
|A^\varepsilon R|^2\,d\nu.
\end{align*}
Using
\[
|A^\varepsilon R|
\le
\|A^\varepsilon\|_{\mathrm{op}}|R|,
\]
we obtain
\begin{align*}
\|A^\varepsilon R\|_\lambda^2
&\le
K^2\|A^\varepsilon\|_\lambda^2
+
M^2
\int_{\{|R|>K\}}|R|^2\,d\nu.
\end{align*}
Since \(\|R\|_\lambda<\infty\), the second term can be made
arbitrarily small by choosing \(K\) sufficiently large. For this fixed
\(K\), the first term tends to $0$ as \(\varepsilon\to0\). Hence
\[
\|A^\varepsilon R\|_\lambda\longrightarrow0\ \ \ {\rm as}\ \varepsilon\to0.
\]

Moreover, if \(\{R^\varepsilon\}_{0<\varepsilon\leq1}\subset L_{\mathcal F}^{2,\lambda}\) and \(R^\varepsilon\longrightarrow R\) in \(L_{\mathcal F}^{2,\lambda}\) as $\varepsilon\to0$, we have
\[\|A^\varepsilon R^\varepsilon\|_{\lambda}^2\le 2M^2\|R^\varepsilon-R\|_{\lambda}^2+2\|A^\varepsilon R\|_{\lambda}^2\longrightarrow0\ \ \ {\rm as}\ \varepsilon\to0.\]

We turn to the proof of \emph{(ii)}. In the finite product measure space
\((\Omega\times\mathbb Z_+\times[0,1],\nu\otimes d\rho)\), for $0\leq\rho\leq1$, set
\[
S_t^{\varepsilon,\rho}
:=S_t+\rho(S_t^\varepsilon-S_t)
\qquad {\rm and}\qquad
G^\varepsilon(t,\rho)
:=\Phi(t,S_t^{\varepsilon,\rho})-\Phi(t,S_t).
\]
Since
\[
\|S^\varepsilon-S\|_\lambda\longrightarrow0,
\]
we have
\[
\int_0^1
\|S^{\varepsilon,\rho}-S\|_\lambda^2\,d\rho
=
\int_0^1\rho^2\,d\rho\,
\|S^\varepsilon-S\|_\lambda^2
\longrightarrow0
\qquad {\rm as}\ \varepsilon\to0.
\]
Thus \(S^{\varepsilon,\rho}\longrightarrow S\) in
\((\nu\otimes d\rho)\)-measure as $\varepsilon\to0$. By the a.s. continuity of
\(\Phi(t,\cdot)\), it follows that
\(
G^\varepsilon\longrightarrow0
\) in \((\nu\otimes d\rho)\)-measure as $\varepsilon\to0$.

Now let $\{\varepsilon_n\}_{n\ge1}$ be an arbitrary sequence such that
$\varepsilon_n\to0$ as $n\to\infty$.
Since as
\[
\|S^{\varepsilon_n}-S\|_\lambda^2
=
\sum_{t=0}^{\infty}
\E\left[
e^{-\lambda t}
|S_t^{\varepsilon_n}-S_t|_{\mathcal X}^2
\right]
\longrightarrow0\ \ \ {\rm as}\ n\to\infty,
\]
we have
\[
|S_t^{\varepsilon_n}-S_t|_{\mathcal X}^2
\longrightarrow0
\qquad\text{in }L^1(\nu)\ \ \ {\rm as}\ n\to\infty.
\]
Hence the sequence
\(
\left\{
|S_t^{\varepsilon_n}-S_t|_{\mathcal X}^2
\right\}_{n\ge1}
\)
is uniformly integrable with respect to \(\nu\).
Moreover, by the linear growth condition,
\begin{align*}
|G^{\varepsilon_n}(t,\rho)|_{\mathcal Y}^2\leq
C\left(
a_t^2+1
+|S_t^{\varepsilon_n,\rho}|_{\mathcal X}^2
+|S_t|_{\mathcal X}^2
\right)\leq
C\left(
a_t^2+1
+|S_t|_{\mathcal X}^2
+|S_t^{\varepsilon_n}-S_t|_{\mathcal X}^2
\right),
\end{align*}
where \(C>0\) is independent of \(n\) and \(\rho\).
Since \(a^2\), \(1\), and \(|S_t|_{\mathcal X}^2\) are
\(\nu\)-integrable, and
\(
\left\{
|S_t^{\varepsilon_n}-S_t|_{\mathcal X}^2
\right\}_{n\in\mathbb{N}}
\)
is uniformly integrable with respect to \(\nu\), while the
right-hand side of the above estimate is independent of \(\rho\),
the same uniform-integrability property is preserved after lifting
these functions to
\((\Omega\times\mathbb Z_+\times[0,1],\nu\otimes d\rho)\).
Therefore,
\(
\left\{
|G^{\varepsilon_n}|_{\mathcal Y}^2
\right\}_{n\in\mathbb{N}}
\)
is uniformly integrable with respect to
\(\nu\otimes d\rho\).


Since \(G^{\varepsilon_n}\longrightarrow0\) in
\((\nu\otimes d\rho)\)-measure  as \(n\to \infty\), we also have
\(
|G^{\varepsilon_n}|_{\mathcal Y}^2\longrightarrow0
\)
in \((\nu\otimes d\rho)\)-measure  as \(n\to\infty\). Therefore, by Vitali's theorem,
\[
\|G^{\varepsilon_n}\|_{L^2(\nu\otimes d\rho)}^2
=
\int_0^1
\sum_{t=0}^{\infty}
\E\left[
e^{-\lambda t}
|G^{\varepsilon_n}(t,\rho)|_{\mathcal Y}^2
\right]d\rho
\longrightarrow0,\ \ \  {\rm as}\ n\to\infty.
\]

Since $\{\varepsilon_n\}_{n\ge1}$ was arbitrary, we conclude that
\(
\|G^\varepsilon\|_{L^2(\nu\otimes d\rho)}^2
\longrightarrow0
\qquad\text{as }\varepsilon\to0.
\)

Finally, by the triangle inequality and the Cauchy--Schwarz
inequality on \([0,1]\),
\begin{align*}
|\bar\Phi^\varepsilon(t)-\Phi(t,S_t)|_{\mathcal Y}^2
=
\left|
\int_0^1
G^\varepsilon(t,\rho)\,d\rho
\right|_{\mathcal Y}^2\leq
\left(
\int_0^1
|G^\varepsilon(t,\rho)|_{\mathcal Y}\,d\rho
\right)^2\leq
\int_0^1
|G^\varepsilon(t,\rho)|_{\mathcal Y}^2\,d\rho.
\end{align*}
Therefore, as \(\varepsilon\to0\),
\begin{align*}
\|\bar\Phi^\varepsilon-\Phi(\cdot,S)\|_\lambda^2
&=
\sum_{t=0}^{\infty}
\E\left[
e^{-\lambda t}
|\bar\Phi^\varepsilon(t)-\Phi(t,S_t)|_{\mathcal Y}^2
\right]\\&\leq
\int_0^1
\sum_{t=0}^{\infty}
\E\left[
e^{-\lambda t}
|G^\varepsilon(t,\rho)|_{\mathcal Y}^2
\right]d\rho=
\|G^\varepsilon\|_{L^2(\nu\otimes d\rho)}^2
\longrightarrow0,
\end{align*}
which proves the desired convergence.

\end{proof}

In the rest of this section, for \(\bar u, u\in\Uad\), we set \(v=u-\bar u\) and \(u^\varepsilon=\bar u+\varepsilon v\), \(0<\varepsilon\leq1\). Then we denote by \((\bar X,\bar Y,\bar Z,\bar N)\) and \((X^\varepsilon,Y^\varepsilon,Z^\varepsilon,N^\varepsilon)\) the states corresponding to controls $\bar u$ and $u^\varepsilon$, respectively.

\begin{lemma}[Stability under control perturbations]\label{lem:state-stability}
Under Assumptions~\ref{ass:control-state}--\ref{ass:control-origin}, there exists \(C>0\) such that
\[
\sum_{t=0}^{\infty}\E\left[e^{-\lambda t}
\left(|X_t^\varepsilon-\bar X_t|^2+|Y_t^\varepsilon-\bar Y_t|^2+|Z_t^\varepsilon-\bar Z_t|^2\right)\right]
\leq
C\varepsilon^2\sum_{t=0}^{\infty}\E[e^{-\lambda t}|v_t|^2].
\]
\end{lemma}

\begin{proof}
Set
\[
\widehat X_t=X_t^\varepsilon-\bar X_t,
\qquad
\widehat Y_t=Y_t^\varepsilon-\bar Y_t,
\qquad
\widehat Z_t=Z_t^\varepsilon-\bar Z_t,
\qquad
\Delta\widehat N_t=\Delta N_t^\varepsilon-\Delta\bar N_t.
\]
Then the controlled system is written as
\begin{equation}\label{eq:state-stability-difference-forward}
\widehat X_{t+1}
=\widehat X_t+\widehat b_t+\widehat\sigma_t\Delta W_t,
\end{equation}
and
\begin{equation}\label{eq:state-stability-difference-backward}
\widehat Y_t+\widehat Z_t\Delta W_t+\Delta\widehat N_t
=e^{-\lambda}\bigl(\widehat Y_{t+1}+\widehat f_{t+1}\bigr),
\end{equation}
where
\[
\widehat b_t=b(t,X_t^\varepsilon,Y_t^\varepsilon,Z_t^\varepsilon,u_t^\varepsilon)
-b(t,\bar X_t,\bar Y_t,\bar Z_t,\bar u_t),
\]
and analogous definitions are used for \(\widehat\sigma_t\) and \(\widehat f_t\).

We split each coefficient difference into a homogeneous state part and a pure control perturbation. For this, define
\[
\mathcal B_t(\widehat x,\widehat y,\widehat z)
:={}b(t,\bar X_t+\widehat x,\bar Y_t+\widehat y,\bar Z_t+\widehat z,\bar u_t)
-b(t,\bar X_t,\bar Y_t,\bar Z_t,\bar u_t)
\]
and
\[
\varphi_t^\varepsilon
:={}b(t,X_t^\varepsilon,Y_t^\varepsilon,Z_t^\varepsilon,u_t^\varepsilon)
-b(t,X_t^\varepsilon,Y_t^\varepsilon,Z_t^\varepsilon,\bar u_t),
\]
and in a similar way, define \(\Sigma^0_t\), \(\psi_t^\varepsilon\) and \(\mathcal G_t\), $\theta_t^\varepsilon$. Then
\[\widehat b_t=\mathcal B_t(\widehat X_t,\widehat Y_t,\widehat Z_t)+\varphi_t^\varepsilon,
\qquad
\widehat\sigma_t=\Sigma^0_t(\widehat X_t,\widehat Y_t,\widehat Z_t)+\psi_t^\varepsilon,
\qquad
\widehat f_t=\mathcal G_t(\widehat X_t,\widehat Y_t,\widehat Z_t)+\theta_t^\varepsilon.
\]
The maps \(\mathcal B_t,\Sigma^0_t,\mathcal G_t\) vanish at the origin and satisfy the same state Lipschitz bounds as \(b,\sigma,f\). Moreover, by the Lipschitz condition,
\[
|\varphi_t^\varepsilon|\leq L_b^u\varepsilon |v_t|,
\qquad
|\psi_t^\varepsilon|\leq L_\sigma^u\varepsilon |v_t|,
\qquad
|\theta_t^\varepsilon|\leq L_f^u\varepsilon |v_t|.
\]
Thus
\[
\|\varphi^\varepsilon\|_\lambda^2+\|\psi^\varepsilon\|_\lambda^2+\|\theta^\varepsilon\|_\lambda^2
\leq C\varepsilon^2\|v\|_\lambda^2.
\]
Equations \eqref{eq:state-stability-difference-forward}--\eqref{eq:state-stability-difference-backward} are exactly of the form \eqref{eq:nonhomogeneous-fbsde} with $0$ initial condition. Applying Corollary~\ref{cor:nonhomogeneous-stability}, we have
\[
\|\widehat X\|_\lambda^2+\|\widehat Y\|_\lambda^2+\|\widehat Z\|_\lambda^2
\leq C\varepsilon^2\|v\|_\lambda^2.
\]
\end{proof}

\subsection{Variational equation and first-order expansion}

Consider the first-order variational equation
\begin{equation}\label{eq:variational}
\left\{
\begin{aligned}
&\xi_{t+1}
={}\xi_t+b_x(t)\xi_t+b_y(t)\eta_t+b_z(t)\zeta_t+b_u(t)v_t+\bigl(\sigma_x(t)\xi_t+\sigma_y(t)\eta_t+\sigma_z(t)\zeta_t+\sigma_u(t)v_t\bigr)\Delta W_t,\\
&\eta_t+\zeta_t\Delta W_t+\Delta Q_t
={}e^{-\lambda}\bigl(\eta_{t+1}+f_x(t+1)\xi_{t+1}+f_y(t+1)\eta_{t+1}+f_z(t+1)\zeta_{t+1}+f_u(t+1)v_{t+1}\bigr),\\
&\xi_0={}0,
\end{aligned}
\right.
\end{equation}
where, for \(g=b,\sigma\), the derivatives \(g_\mu(t)\), \(\mu=x,y,z,u\), are evaluated at \((\bar X_t,\bar Y_t,\bar Z_t,\bar u_t)\), and the derivatives \(f_\mu(t+1)\) are evaluated at \((\bar X_{t+1},\bar Y_{t+1},\bar Z_{t+1},\bar u_{t+1})\).

\begin{lemma}[Solvability of the variational equation]\label{lem:var-solvability}
Under Assumptions~\ref{ass:control-state}--\ref{ass:control-origin}, the first-order variational equation \eqref{eq:variational} admits a unique solution
$(\xi,\eta,\zeta,Q)$.
Moreover, the solution satisfies the stronger summability property
\[
\sum_{t=0}^{\infty}\E\left[e^{-\lambda t}(|\xi_t|^2+|\eta_t|^2+|\zeta_t|^2)\right]
\leq
C\sum_{t=0}^{\infty}\E[e^{-\lambda t}|v_t|^2].
\]
\end{lemma}

\begin{proof}
The variational equation is a linear nonhomogeneous FBS\(\Delta\)E. Its homogeneous coefficient maps are
\begin{align*}
&(\xi,\eta,\zeta)\longmapsto b_x(t)\xi+b_y(t)\eta+b_z(t)\zeta,\\
&(\xi,\eta,\zeta)\longmapsto \sigma_x(t)\xi+\sigma_y(t)\eta+\sigma_z(t)\zeta,\\
&(\xi,\eta,\zeta)\longmapsto f_x(t)\xi+f_y(t)\eta+f_z(t)\zeta.
\end{align*}
These maps all vanish at the origin. Since the first derivatives of \(b,\sigma,f\) are uniformly bounded by the Lipschitz constants in Assumption~\ref{ass:control-state}, the homogeneous part satisfies exactly the assumptions in Corollary~\ref{cor:nonhomogeneous-stability}. For the nonhomogeneous terms
$\varphi_t=b_u(t)v_t$,
$\psi_t=\sigma_u(t)v_t$ and
$\theta_t=f_u(t)v_t$.
Since \(b_u,\sigma_u,f_u\) are uniformly bounded,
\[
\|\varphi\|_\lambda^2+\|\psi\|_\lambda^2+\|\theta\|_\lambda^2
\leq C\|v\|_\lambda^2.
\]
Thus, by Corollary~\ref{cor:nonhomogeneous-stability}, we have the existence and uniqueness of solution and the estimate
\[
\|\xi\|_\lambda^2+\|\eta\|_\lambda^2+\|\zeta\|_\lambda^2
\leq C\|v\|_\lambda^2.
\]
\end{proof}

\begin{lemma}[Remainder estimate]\label{lem:remainder}
Under Assumptions~\ref{ass:control-state}--\ref{ass:control-diff-cost}, if \((\xi,\eta,\zeta,Q)\) is the solution to FBS\(\Delta\)E \eqref{eq:variational}, then
\begin{align*}
\lim_{\varepsilon\downarrow0}
\sum_{t=0}^{\infty}\E\biggl[e^{-\lambda t}\biggl(
\left|\frac{X_t^\varepsilon-\bar X_t}{\varepsilon}-\xi_t\right|^2
+\left|\frac{Y_t^\varepsilon-\bar Y_t}{\varepsilon}-\eta_t\right|^2+\left|\frac{Z_t^\varepsilon-\bar Z_t}{\varepsilon}-\zeta_t\right|^2
\biggr)\biggr]=0,
\end{align*}
or equivalently,
\[
X^\varepsilon-\bar X=\varepsilon\xi+o(\varepsilon),\qquad
Y^\varepsilon-\bar Y=\varepsilon\eta+o(\varepsilon),\qquad
Z^\varepsilon-\bar Z=\varepsilon\zeta+o(\varepsilon)
\]
in $L^{2,\lambda}_{\F}$ space.
\end{lemma}

\begin{proof}
For $t\geq0$, set
\[
R_t^{X,\varepsilon}:=\frac{X_t^\varepsilon-\bar X_t}{\varepsilon},
\qquad
R_t^{Y,\varepsilon}:=\frac{Y_t^\varepsilon-\bar Y_t}{\varepsilon},
\qquad
R_t^{Z,\varepsilon}:=\frac{Z_t^\varepsilon-\bar Z_t}{\varepsilon},
\]
and
\begin{equation}\label{cz3}
\widetilde X_t^\varepsilon:=R_t^{X,\varepsilon}-\xi_t,
\ \ \
\widetilde Y_t^\varepsilon:=R_t^{Y,\varepsilon}-\eta_t,
\ \ \
\widetilde Z_t^\varepsilon:=R_t^{Z,\varepsilon}-\zeta_t,
\ \ \
\Delta\widetilde N_t^\varepsilon
:=\frac{\Delta N_t^\varepsilon-\Delta\bar N_t}{\varepsilon}-\Delta Q_t.
\end{equation}
By Lemma~\ref{lem:state-stability},
\[
\|R^{X,\varepsilon}\|_\lambda+
\|R^{Y,\varepsilon}\|_\lambda+
\|R^{Z,\varepsilon}\|_\lambda
\le C\|v\|_\lambda,
\]
and as $\varepsilon\to0$,
\[
(X^\varepsilon,Y^\varepsilon,Z^\varepsilon,u^\varepsilon)
\longrightarrow
(\bar X,\bar Y,\bar Z,\bar u)
\quad\text{in }L^{2,\lambda}.
\]
To show the convergence of the averaged derivatives more clear, for $t\geq0$, write
\[
S_t:=(\bar X_t,\bar Y_t,\bar Z_t,\bar u_t),
\qquad
S_t^\varepsilon:=(X_t^\varepsilon,Y_t^\varepsilon,Z_t^\varepsilon,u_t^\varepsilon),
\]
and, for \(s\in[0,1]\), set
\[
S_t^{\varepsilon,s}:=S_t+s(S_t^\varepsilon-S_t).
\]
Then
\[
\int_0^1\|S^{\varepsilon,s}-S\|_{\lambda}^2\,ds
\leq
\|S^\varepsilon-S\|_{\lambda}^2\longrightarrow0\ \ \ {\rm as}\ \varepsilon\to0.
\]
Hence \(S^{\varepsilon,s}\longrightarrow S\) in measure as $\varepsilon\to0$ in the product space
\(\Omega\times\mathbb Z_+\times[0,1]\).

For \(g=b,\sigma\) and \(\mu=x,y,z,u\), define averaged derivatives
\begin{align*}
\bar g_\mu^\varepsilon(t)
:=\int_0^1 g_\mu\bigl(t,\bar X_t+s(X_t^\varepsilon-\bar X_t),
\bar Y_t+s(Y_t^\varepsilon-\bar Y_t),
\bar Z_t+s(Z_t^\varepsilon-\bar Z_t),
\bar u_t+s(u_t^\varepsilon-\bar u_t)\bigr)\,ds,
\end{align*}
and we have
\begin{align*}
\frac{b(t,X_t^\varepsilon,Y_t^\varepsilon,Z_t^\varepsilon,u_t^\varepsilon)
-b(t,\bar X_t,\bar Y_t,\bar Z_t,\bar u_t)}{\varepsilon}=\bar b_x^\varepsilon(t)R_t^{X,\varepsilon}
+\bar b_y^\varepsilon(t)R_t^{Y,\varepsilon}
+\bar b_z^\varepsilon(t)R_t^{Z,\varepsilon}
+\bar b_u^\varepsilon(t)v_t.
\end{align*}
Subtracting the variational equation, by \eqref{cz3}, we have
\begin{align*}
\widetilde X_{t+1}^\varepsilon
={}&\widetilde X_t^\varepsilon+\bar b_x^\varepsilon(t)\widetilde X_t^\varepsilon
+\bar b_y^\varepsilon(t)\widetilde Y_t^\varepsilon
+\bar b_z^\varepsilon(t)\widetilde Z_t^\varepsilon+\rho_t^{b,\varepsilon}+\bigl(\bar\sigma_x^\varepsilon(t)\widetilde X_t^\varepsilon
+\bar\sigma_y^\varepsilon(t)\widetilde Y_t^\varepsilon
+\bar\sigma_z^\varepsilon(t)\widetilde Z_t^\varepsilon
+\rho_t^{\sigma,\varepsilon}\bigr)\Delta W_t,
\end{align*}
where
\[
\rho_t^{b,\varepsilon}
=(\bar b_x^\varepsilon(t)-b_x(t))\xi_t
+(\bar b_y^\varepsilon(t)-b_y(t))\eta_t
+(\bar b_z^\varepsilon(t)-b_z(t))\zeta_t
+(\bar b_u^\varepsilon(t)-b_u(t))v_t,
\]
and \(\rho_t^{\sigma,\varepsilon}\) is defined analogously. Similarly, for \(\mu=x,y,z,u\), define \(\bar f_\mu^\varepsilon(t+1)\) along the time \(t+1\) segment, and we have
\begin{align*}
\widetilde Y_t^\varepsilon+\widetilde Z_t^\varepsilon\Delta W_t+\Delta\widetilde N_t^\varepsilon
=e^{-\lambda}\bigl(&\widetilde Y_{t+1}^\varepsilon
+\bar f_x^\varepsilon(t+1)\widetilde X_{t+1}^\varepsilon
+\bar f_y^\varepsilon(t+1)\widetilde Y_{t+1}^\varepsilon+\bar f_z^\varepsilon(t+1)\widetilde Z_{t+1}^\varepsilon
+\rho_{t+1}^{f,\varepsilon}\bigr),
\end{align*}
where
\begin{align*}
\rho_{t+1}^{f,\varepsilon}
={}&(\bar f_x^\varepsilon(t+1)-f_x(t+1))\xi_{t+1}
+(\bar f_y^\varepsilon(t+1)-f_y(t+1))\eta_{t+1}\\
&+(\bar f_z^\varepsilon(t+1)-f_z(t+1))\zeta_{t+1}
+(\bar f_u^\varepsilon(t+1)-f_u(t+1))v_{t+1}.
\end{align*}

For \(g=b,\sigma,f\) and \(\mu=x,y,z,u\), apply
Lemma~\ref{lem:controlled-convergence} (ii) with
\[
S_t=(\bar X_t,\bar Y_t,\bar Z_t,\bar u_t),
\ \ \
S_t^\varepsilon=(X_t^\varepsilon,Y_t^\varepsilon,
Z_t^\varepsilon,u_t^\varepsilon),\ \ \ \Phi=g_\mu.
\]
Since the derivatives of \(b,\sigma,f\) are a.s.
continuous and uniformly bounded, the conditions of
Lemma~\ref{lem:controlled-convergence} (ii) are satisfied. Hence
\[
\|\bar g_\mu^\varepsilon-g_\mu\|_{\lambda}^2
\longrightarrow0\ \ \ {\rm as}\ \varepsilon\to0.
\]
and
\[
\sup_{0<\varepsilon\le1}
\operatorname*{ess\,sup}_{(\omega,t)}
\|\bar g_\mu^\varepsilon(t)-g_\mu(t)\|_{\mathrm{op}}
<\infty.
\]

For instance, taking
\[
A^\varepsilon=\bar b_x^\varepsilon-b_x\ \ \ {\rm and}\ \ \
R=\xi
\]
in Lemma~\ref{lem:controlled-convergence} (i), we obtain
\[
\|(\bar b_x^\varepsilon-b_x)\xi\|_{\lambda}^2
\longrightarrow0\ \ \ {\rm as}\ \varepsilon\to0.
\]
All the other terms in
\(\rho^{b,\varepsilon}\), \(\rho^{\sigma,\varepsilon}\), and
\(\rho^{f,\varepsilon}\) can be dealt with in a similar way. Hence
\[
\|\rho^{b,\varepsilon}\|_{\lambda}^2
+
\|\rho^{\sigma,\varepsilon}\|_{\lambda}^2
+
\sum_{s=1}^{\infty}
\E\!\left[
e^{-\lambda s}|\rho_s^{f,\varepsilon}|^2
\right]
\longrightarrow0\ \ \ {\rm as}\ \varepsilon\to0.
\]
Since
\[
\sum_{t=0}^{\infty}
\E\!\left[
e^{-\lambda t}|\rho_{t+1}^{f,\varepsilon}|^2
\right]
=
e^\lambda
\sum_{s=1}^{\infty}
\E\!\left[
e^{-\lambda s}|\rho_s^{f,\varepsilon}|^2
\right],
\]
we conclude that
\[
\|\rho^{b,\varepsilon}\|_{\lambda}^2
+
\|\rho^{\sigma,\varepsilon}\|_{\lambda}^2
+
\sum_{t=0}^{\infty}
\E\!\left[
e^{-\lambda t}|\rho_{t+1}^{f,\varepsilon}|^2
\right]
\longrightarrow0\ \ \ {\rm as}\ \varepsilon\to0.
\]

The homogeneous coefficients
\(\bar b_\mu^\varepsilon,\bar\sigma_\mu^\varepsilon,
\bar f_\mu^\varepsilon\), \(\mu=x,y,z\), satisfy the same bounds as
the corresponding derivatives of \(b,\sigma,f\). Hence the same
discount condition remains valid for the homogeneous part.
Applying Corollary~\ref{cor:nonhomogeneous-stability} to the system for
\((\widetilde X^\varepsilon,\widetilde Y^\varepsilon,
\widetilde Z^\varepsilon)\), we obtain as $\varepsilon\to0$,
\[
\begin{aligned}
	\|\widetilde X^\varepsilon\|_{\lambda}^2
	+\|\widetilde Y^\varepsilon\|_{\lambda}^2
	+\|\widetilde Z^\varepsilon\|_{\lambda}^2\le
	C\left(
	\|\rho^{b,\varepsilon}\|_{\lambda}^2
	+\|\rho^{\sigma,\varepsilon}\|_{\lambda}^2
	+\sum_{t=0}^{\infty}
	\E\!\left[
	e^{-\lambda t}|\rho_{t+1}^{f,\varepsilon}|^2
	\right]
	\right)
	\longrightarrow0.
\end{aligned}
\]
The proof is complete.

\end{proof}

\begin{lemma}[First-order expansion of the cost]\label{lem:cost-expansion}
Under Assumptions~\ref{ass:control-state}--\ref{ass:control-diff-cost},
\begin{align*}
\lim_{\varepsilon\downarrow0}\frac{J(u^\varepsilon)-J(\bar u)}{\varepsilon}
=
\sum_{t=0}^{\infty}\E\bigl[e^{-\lambda t}(&\ip{l_x(t)}{\xi_t}+\ip{l_y(t)}{\eta_t}+\ip{l_z(t)}{\zeta_t}
+\ip{l_u(t)}{v_t})\bigr],
\end{align*}
where all derivatives of \(l\) are evaluated along the tuple \((\bar X_t,\bar Y_t,\bar Z_t,\bar u_t)\).
\end{lemma}

\begin{proof}
For \(\mu=x,y,z,u\), define the averaged derivative
\[
\bar l_\mu^\varepsilon(t):=
\int_0^1 l_\mu\bigl(t,\bar X_t+s(X_t^\varepsilon-\bar X_t),
\bar Y_t+s(Y_t^\varepsilon-\bar Y_t),
\bar Z_t+s(Z_t^\varepsilon-\bar Z_t),
\bar u_t+s(u_t^\varepsilon-\bar u_t)\bigr)\,ds.
\]
The fundamental theorem of calculus along the line segment yields
\begin{align}\label{cz4}
&\frac{l(t,X_t^\varepsilon,Y_t^\varepsilon,Z_t^\varepsilon,u_t^\varepsilon)
-l(t,\bar X_t,\bar Y_t,\bar Z_t,\bar u_t)}{\varepsilon} \notag\\
&=\ip{\bar l_x^\varepsilon(t)}{\frac{X_t^\varepsilon-\bar X_t}{\varepsilon}}
+\ip{\bar l_y^\varepsilon(t)}{\frac{Y_t^\varepsilon-\bar Y_t}{\varepsilon}}
+\ip{\bar l_z^\varepsilon(t)}{\frac{Z_t^\varepsilon-\bar Z_t}{\varepsilon}}
+\ip{\bar l_u^\varepsilon(t)}{v_t}.
\end{align}
By Lemma~\ref{lem:state-stability},
\[
(X^\varepsilon,Y^\varepsilon,Z^\varepsilon,u^\varepsilon)\longrightarrow
(\bar X,\bar Y,\bar Z,\bar u)\ \ \ \text{in }L^{2,\lambda}\ {\rm as}\ \varepsilon\to0.
\]
The linear growth condition on \(l_x,l_y,l_z,l_u\) and Lemma~\ref{lem:controlled-convergence} (ii) implies
\[
\|\bar l_\mu^\varepsilon-l_\mu\|_{\lambda}\longrightarrow0\ \ \ {\rm as}\ \varepsilon\to0,
\]
where \(l_\mu(t)=l_\mu(t,\bar X_t,\bar Y_t,\bar Z_t,\bar u_t)\). In particular, for each \(\mu\) there exist \(\varepsilon_0>0\) and \(C_\mu<\infty\) such that
\[
\sup_{0<\varepsilon\leq\varepsilon_0}
\|\bar l_\mu^\varepsilon\|_{\lambda}\leq C_\mu.
\]
Moreover,
by Lemma~\ref{lem:remainder}, as \(\varepsilon\to0\),
\[
\frac{X^\varepsilon-\bar X}{\varepsilon}\longrightarrow\xi,\qquad
\frac{Y^\varepsilon-\bar Y}{\varepsilon}\longrightarrow\eta,\qquad
\frac{Z^\varepsilon-\bar Z}{\varepsilon}\longrightarrow\zeta
\quad\text{in }L^{2,\lambda}.
\]

Then we show that the first term of the right hand side of \eqref{cz4} converges in \(L^1(\nu)\) to
$\ip{l_x}{\xi}$. To see this, as \(\varepsilon\to0\),
\begin{align*}
\sum_{t=0}^{\infty}\E\biggl[e^{-\lambda t}\biggl|
\ip{\bar l_x^\varepsilon(t)}{\frac{X_t^\varepsilon-\bar X_t}{\varepsilon}}
-\ip{l_x(t)}{\xi_t}\biggr|\biggr]\le
\|\bar l_x^\varepsilon\|_{\lambda}
\left\|\frac{X^\varepsilon-\bar X}{\varepsilon}-\xi\right\|_{\lambda}
+\|\bar l_x^\varepsilon-l_x\|_{\lambda}\|\xi\|_{\lambda}
\longrightarrow0.
\end{align*}
The convergence of the second and the third terms can be deduced similarly. As for the control term in \eqref{cz4}, as \(\varepsilon\to0\),
\[
\sum_{t=0}^{\infty}\E\left[e^{-\lambda t}
\left|\ip{\bar l_u^\varepsilon(t)-l_u(t)}{v_t}\right|\right]
\le
\|\bar l_u^\varepsilon-l_u\|_{\lambda}\|v\|_{\lambda}\longrightarrow0.
\]
Thus \eqref{cz4} converges in \(L^1(\nu)\) to
\[
\ip{l_x}{\xi}+\ip{l_y}{\eta}+\ip{l_z}{\zeta}+\ip{l_u}{v}.
\]
Then Lemma \ref{lem:cost-expansion} follows by taking the integration  in \(\Omega\times\mathbb Z_+\) with respect to the measure \(\nu\).
\end{proof}

\subsection{Adjoint equation and transversality relations}

Define the Hamiltonian by
\begin{equation}\label{eq:Hamiltonian}
H(t,x,y,z,u,p,q,k)
:=
l(t,x,y,z,u)-\ip{b(t,x,y,z,u)}{p}-\ip{\sigma(t,x,y,z,u)}{q}
+\ip{f(t,x,y,z,u)}{k},
\end{equation}
where \(p\in\R^m\), \(q\in\R^{m\times d}\) and \(k\in\R^n\). The inner product between matrices is the Frobenius inner product. In the following, all derivatives of \(H\) are evaluated along the optimal trajectory
\[
(\bar X_t,\bar Y_t,\bar Z_t,\bar u_t,p_t,q_t,k_t)
\]
unless otherwise stated.

We use Fr\'{e}chet derivatives with respect to the matrix variable \(z\). For example,
\[
b_z(t):\R^{n\times d}\to\R^m
\]
is a bounded linear operator, and \(b_z(t)^*:\R^m\to\R^{n\times d}\) represents its adjoint under the Euclidean and Frobenius inner products, namely
\[
\langle b_z(t)h,p\rangle_{\R^m}=\langle h,b_z(t)^*p\rangle_{\R^{n\times d}}.
\]

The adjoint equation associated with the optimal trajectory is
\begin{equation}\label{eq:adjoint}
\left\{
\begin{aligned}
&k_{t+1}
={}k_t+H_y(t)+H_z(t)\Delta W_t,\\
&p_t+q_t\Delta W_t+\Delta R_t
={}e^{-\lambda}\left(p_{t+1}-H_x(t+1)\right),\\
&k_0={}0.
\end{aligned}
\right.
\end{equation}
Equivalently,
\begin{align*}
k_{t+1}
={}&k_t+\bigl(l_y(t)-b_y(t)^* p_t-\sigma_y(t)^* q_t+f_y(t)^* k_t\bigr)\\
&+\bigl(l_z(t)-b_z(t)^* p_t-\sigma_z(t)^* q_t+f_z(t)^* k_t\bigr)\Delta W_t,
\end{align*}
and
\begin{align*}
p_t+q_t\Delta W_t+\Delta R_t
={}e^{-\lambda}\Bigl(p_{t+1}-l_x(t+1)+b_x(t+1)^* p_{t+1}+\sigma_x(t+1)^* q_{t+1}-f_x(t+1)^* k_{t+1}\Bigr).
\end{align*}

We next show that the adjoint equation is well posed. No additional coefficient assumptions on the adjoint equation are needed. Indeed, the adjoint equation is a linear fully coupled FBS\(\Delta\)E whose coefficients are composed of the first-order derivatives of the original coefficients \(b,\sigma,f\).
By Assumption~\ref{ass:control-state}, its coefficients are uniformly bounded
and satisfy the Lipschitz conditions in
Theorem~\ref{thm:fbsde-solvability}.
The only nonhomogeneous terms in the adjoint equation are \(l_x,l_y,l_z\), whose weighted square-integrability follows from the growth condition on \(l\) and the weighted square-integrability of the optimal state-control process. Assumption~\ref{ass:control-discount} is written so that the same discount parameter \(\lambda\) controls both the state FBS\(\Delta\)E and the adjoint FBS\(\Delta\)E.

\begin{proposition}[Solvability of the adjoint equation]\label{prop:adjoint-solvability}
Under Assumptions~\ref{ass:control-state}--\ref{ass:control-diff-cost}, if \((\bar X,\bar Y,\bar Z,\bar u)\) is the optimal state-control process, then the adjoint equation \eqref{eq:adjoint} admits a unique adapted solution, denoted by $(k,p,q,R)$. Moreover, the solution satisfies the stronger summability property
\[
(k,p,q,\Delta R)
\in
L^{2,\lambda}_{\F}(0,\infty;\R^n)
\times
L^{2,\lambda}_{\F}(0,\infty;\R^m)
\times
L^{2,\lambda}_{\F}(0,\infty;\R^{m\times d})
\times
L^{2,\lambda}_{\F}(0,\infty;\R^m).
\]
\end{proposition}
\begin{proof}
Define
\begin{align*}
&\widehat b(t,k,p,q)
:=l_y(t)-b_y(t)^* p-\sigma_y(t)^* q+f_y(t)^* k,\\
&\widehat\sigma(t,k,p,q)
:=l_z(t)-b_z(t)^* p-\sigma_z(t)^* q+f_z(t)^* k,\\
&\widehat f(t,k,p,q)
:=-l_x(t)+b_x(t)^* p+\sigma_x(t)^* q-f_x(t)^* k.
\end{align*}
Then the adjoint equation \eqref{eq:adjoint} can be written as a fully coupled linear FBS\(\Delta\)E
\begin{equation}\label{eq:adjoint-linear-fbsde}
\left\{
\begin{aligned}
&k_{t+1}=k_t+\widehat b(t,k_t,p_t,q_t)+\widehat\sigma(t,k_t,p_t,q_t)\Delta W_t,\\
&p_t+q_t\Delta W_t+\Delta R_t
=e^{-\lambda}\left[p_{t+1}+\widehat f(t+1,k_{t+1},p_{t+1},q_{t+1})\right],\\
&k_0=0.
\end{aligned}
\right.
\end{equation}

We verify that FBS\(\Delta\)E \eqref{eq:adjoint-linear-fbsde} satisfies the conditions of Theorem \ref{thm:fbsde-solvability}.
First of all, the coefficients \(\widehat b, \widehat\sigma, \widehat f\) satisfy
\[
\begin{aligned}
&|\widehat b(t,k_1,p_1,q_1)-\widehat b(t,k_2,p_2,q_2)|\leq
L_f^y|k_1-k_2|+L_b^y|p_1-p_2|+L_\sigma^y|q_1-q_2|,\\
&|\widehat\sigma(t,k_1,p_1,q_1)-\widehat\sigma(t,k_2,p_2,q_2)|\leq
L_f^z|k_1-k_2|+L_b^z|p_1-p_2|+L_\sigma^z|q_1-q_2|,\\
&|\widehat f(t,k_1,p_1,q_1)-\widehat f(t,k_2,p_2,q_2)|\leq
L_f^x|k_1-k_2|+L_b^x|p_1-p_2|+L_\sigma^x|q_1-q_2|.
\end{aligned}
\]
Thus the above estimates show that
\(\widehat b,\widehat\sigma,\widehat f\)
satisfy the Lipschitz condition in
Assumption~\ref{ass:fbsde-coeff}.
It is straightforward to verify, by the monotonicity of the spectral radius for nonnegative matrices,
that Assumption~\ref{ass:control-discount} also guarantees the discount condition required by Theorem~\ref{thm:fbsde-solvability} for the adjoint system.

It remains only to check Assumption \ref{ass:fbsde-origin}, the integrability at the origin for the inhomogeneous terms. At the origin,
\[
\widehat b(t,0,0,0)=l_y(t),\qquad
\widehat\sigma(t,0,0,0)=l_z(t),\qquad
\widehat f(t,0,0,0)=-l_x(t).
\]
By the growth assumption on the derivatives of \(l\),
\[
|l_x(t)|+|l_y(t)|+|l_z(t)|
\leq
\varrho_t+C\bigl(1+|\bar X_t|+|\bar Y_t|+|\bar Z_t|+|\bar u_t|\bigr).
\]
Then due to the weighted square-integrability of \(\varrho\), \(\bar X\), \(\bar Y\), \(\bar Z\) and \(\bar u\), it yields
\[
\sum_{t=0}^{\infty}\E\left[e^{-\lambda t}
\left(|l_x(t)|^2+|l_y(t)|^2+|l_z(t)|^2\right)\right]<\infty.
\]
Hence Assumption~\ref{ass:fbsde-origin} is satisfied.

Therefore, the adjoint equation \eqref{eq:adjoint-linear-fbsde} satisfies all conditions of Theorem~\ref{thm:fbsde-solvability}, and Proposition \ref{prop:adjoint-solvability} follows from Theorem~\ref{thm:fbsde-solvability}.
\end{proof}

\begin{lemma}[Transversality relations]\label{lem:transversality}
Under Assumptions~\ref{ass:control-state}--\ref{ass:control-diff-cost}, if \((\xi,\eta,\zeta,Q)\) and \((k,p,q,R)\) are the solutions to the variational equation \eqref{eq:variational} and the adjoint equation \eqref{eq:adjoint-linear-fbsde}, respectively, then
\[
\lim_{T\to\infty}\E\left[e^{-\lambda T}\ip{\xi_T}{p_T}\right]=0
\ \ \ {\rm and}\ \ \
\lim_{T\to\infty}\E\left[e^{-\lambda T}\ip{\eta_T}{k_T}\right]=0.
\]
\end{lemma}
\begin{proof}
We only prove the first convergence, and the second one can be proved similarly. First note
\[
\left|\E[e^{-\lambda T}\ip{\xi_T}{p_T}]\right|
\leq
\left(\E[e^{-\lambda T}|\xi_T|^2]\right)^{1/2}
\left(\E[e^{-\lambda T}|p_T|^2]\right)^{1/2}.
\]
Noticing \(\xi,p\in L^{2,\lambda}_{\F}(0,\infty)\), we know
\(\E[e^{-\lambda T}|\xi_T|^2]\rightarrow0\) and \(\E[e^{-\lambda T}|p_T|^2]\rightarrow0\) as $T\to\infty$, which leads to the first convergence.
\end{proof}

\begin{remark}\label{rem:transversality-no-assumption}
The transversality relations are not imposed as additional conditions. They follow from the weighted square-integrability of the variational and adjoint processes.
\end{remark}

\subsection{Stochastic maximum principle and verification theorem}

We first record the duality identity used in the proof of the stochastic maximum principle. This identity is the core algebraic step: it transfers the first-order state variation terms in the cost expansion to a first-order term involving only the control variation.

\begin{lemma}[Duality identity]\label{lem:duality-identity}
Under Assumptions~\ref{ass:control-state}--\ref{ass:control-diff-cost}, if  \((\xi,\eta,\zeta,Q)\) and \((k,p,q,R)\) are the solutions to the variational equation \eqref{eq:variational} and the adjoint equation \eqref{eq:adjoint}, respectively, then
\begin{align}\label{eq:duality-identity}
&\sum_{t=0}^{\infty}\E\left[e^{-\lambda t}\left(
\ip{l_x(t)}{\xi_t}+\ip{l_y(t)}{\eta_t}+\ip{l_z(t)}{\zeta_t}
\right)\right] \\
&\qquad=
\sum_{t=0}^{\infty}\E\left[e^{-\lambda t}
\ip{-b_u(t)^* p_t-\sigma_u(t)^* q_t+f_u(t)^* k_t}{v_t}
\right], \notag
\end{align}
where all derivatives are evaluated along the optimal trajectory.
\end{lemma}

\begin{proof}
We prove the identity by applying the discrete summation-by-parts formula to two dual pairings. First we work on a finite interval \(\{0,\ldots,T\}\). For a scalar sequence \(A=\{A_t\}\), we always write \(\Delta A_t:=A_{t+1}-A_t\) in this proof.

\emph{Step 1: Pairing between \(\xi\) and \(p\).}
From the adjoint backward equation, we have
\[
 e^{-\lambda(t+1)}p_{t+1}
=
e^{-\lambda t}(p_t+q_t\Delta W_t+\Delta R_t)
+e^{-\lambda(t+1)}H_x(t+1).
\]
On the other hand, it follows from the variational forward equation that
\[
\xi_{t+1}=\xi_t+B_t+\Sigma_t\Delta W_t,
\]
where
\[
B_t=b_x(t)\xi_t+b_y(t)\eta_t+b_z(t)\zeta_t+b_u(t)v_t
\]
and
\[
\Sigma_t=\sigma_x(t)\xi_t+\sigma_y(t)\eta_t+
\sigma_z(t)\zeta_t+\sigma_u(t)v_t.
\]
Then
\begin{align}\label{eq:xi-p-duality}
\sum_{t=0}^{T}\E\left[\Delta\ip{\xi_t}{e^{-\lambda t}p_t}\right]&=\sum_{t=0}^{T}\E\left[\ip{\xi_{t+1}}{e^{-\lambda(t+1)}p_{t+1}}-
\ip{\xi_t}{e^{-\lambda t}p_t}\right] \\
&=\sum_{t=0}^{T}\E\left[e^{-\lambda t}
\left(\ip{B_t}{p_t}+\ip{\Sigma_t}{q_t}\right)\right]
+\sum_{t=0}^{T}\E\left[e^{-\lambda(t+1)}\ip{\xi_{t+1}}{H_x(t+1)}\right]\notag.
\end{align}
Since \(\xi_0=0\), the left-hand side equals to
$
\E\left[e^{-\lambda(T+1)}\ip{\xi_{T+1}}{p_{T+1}}\right]$,
which tends to $0$ by Lemma~\ref{lem:transversality}.

\emph{Step 2: Pairing between \(\eta\) and \(k\).}
From the variational backward equation,
\[
e^{-\lambda(t+1)}\eta_{t+1}
=e^{-\lambda t}(\eta_t+\zeta_t\Delta W_t+\Delta Q_t)
-e^{-\lambda(t+1)}F_{t+1},
\]
where
\[
F_{t+1}:=f_x(t+1)\xi_{t+1}+f_y(t+1)\eta_{t+1}
+f_z(t+1)\zeta_{t+1}+f_u(t+1)v_{t+1}.
\]
On the other hand, the adjoint forward equation yields
\[
k_{t+1}=k_t+H_y(t)+H_z(t)\Delta W_t.
\]
Thus
\begin{align}\label{eq:eta-k-duality}
\sum_{t=0}^{T}\E\left[\Delta\ip{e^{-\lambda t}\eta_t}{k_t}\right] &=\sum_{t=0}^{T}\E\left[\ip{e^{-\lambda(t+1)}\eta_{t+1}}{k_{t+1}}
-\ip{e^{-\lambda t}\eta_t}{k_t}\right] \\
&=\sum_{t=0}^{T}\E\left[e^{-\lambda t}\left(\ip{\eta_t}{H_y(t)}+\ip{\zeta_t}{H_z(t)}\right)\right]-
\sum_{t=0}^{T}\E\left[e^{-\lambda(t+1)}\ip{F_{t+1}}{k_{t+1}}\right].\notag
\end{align}
Since \(k_0=0\), the left-hand side equals to
$\E\left[e^{-\lambda(T+1)}\ip{\eta_{T+1}}{k_{T+1}}\right]$,
which tends to $0$ by Lemma~\ref{lem:transversality} again.

\emph{Step 3: Passage to the infinite horizon.}
First note that the bounded coefficient derivatives and the weighted square-integrability of \((\xi,\eta,\zeta,v)\) imply \(B,\Sigma,F\in L^{2,\lambda}\). Moreover, the growth condition on the derivatives of \(l\), together with the weighted square-integrability of the state-control and adjoint processes, yields \(H_x,H_y,H_z\in L^{2,\lambda}\). Hence 
\begin{align*}
&\sum_{t=0}^{\infty}\E\bigl[e^{-\lambda t}\bigl(
|\langle B_t,p_t\rangle|+|\langle\Sigma_t,q_t\rangle|
+|\langle\eta_t,H_y(t)\rangle|+|\langle\zeta_t,H_z(t)\rangle|
\bigr)\bigr]\\
&\qquad\le
\|B\|_\lambda\|p\|_\lambda
+\|\Sigma\|_\lambda\|q\|_\lambda
+\|\eta\|_\lambda\|H_y\|_\lambda
+\|\zeta\|_\lambda\|H_z\|_\lambda<\infty,
\end{align*}
and by a similar method,
\[
\sum_{t=1}^{\infty}\E\left[e^{-\lambda t}
\left(|\langle\xi_t,H_x(t)\rangle|+|\langle F_t,k_t\rangle|\right)\right]
\le \|\xi\|_\lambda\|H_x\|_\lambda+\|F\|_\lambda\|k\|_\lambda<\infty.
\]
Based on \eqref{eq:xi-p-duality} and \eqref{eq:eta-k-duality}, by taking $T\to\infty$, we obtain
\begin{align*}
0={}&
\sum_{t=0}^{\infty}\E\left[e^{-\lambda t}
\left(\ip{B_t}{p_t}+\ip{\Sigma_t}{q_t}
+\ip{\eta_t}{H_y(t)}+\ip{\zeta_t}{H_z(t)}\right)\right] \\
&+\sum_{t=1}^{\infty}\E\left[e^{-\lambda t}\ip{\xi_t}{H_x(t)}\right]
-
\sum_{t=1}^{\infty}\E\left[e^{-\lambda t}\ip{F_t}{k_t}\right].
\end{align*}
Since \(\xi_0=0\) and \(k_0=0\), the sums from \(t=1\) may be extended to sums from \(t=0\). Putting
\begin{align*}
H_x=l_x-b_x^* p-\sigma_x^* q+f_x^* k,\\
H_y=l_y-b_y^* p-\sigma_y^* q+f_y^* k,\\
H_z=l_z-b_z^* p-\sigma_z^* q+f_z^* k,
\end{align*}
into above equality, we notice that all terms involving \(b_x,b_y,b_z\), \(\sigma_x,\sigma_y,\sigma_z\), and \(f_x,f_y,f_z\) cancel with the corresponding terms in \(B_t,\Sigma_t\), and \(F_t\). Thus we have
\[
0=
\sum_{t=0}^{\infty}\E\left[e^{-\lambda t}
\left(\ip{l_x(t)}{\xi_t}+\ip{l_y(t)}{\eta_t}+\ip{l_z(t)}{\zeta_t}
+\ip{b_u(t)v_t}{p_t}+\ip{\sigma_u(t)v_t}{q_t}-\ip{f_u(t)v_t}{k_t}
\right)\right],
\]
which leads to \eqref{eq:duality-identity}.
\end{proof}

\begin{theorem}[Stochastic maximum principle]\label{thm:smp}
Under Assumptions~\ref{ass:control-state}--\ref{ass:control-diff-cost}, if \(\bar u\in\Uad\) is an optimal control,
\((\bar X,\bar Y,\bar Z,\bar N)\) is the corresponding optimal state, and \((k,p,q,R)\) is
the solution to the adjoint equation
\eqref{eq:adjoint}, then for any \(u\in\Uad\),
\begin{equation}\label{eq:smp}
\sum_{t=0}^{\infty}
\E\left[
e^{-\lambda t}
\ip{
H_u(t,\bar X_t,\bar Y_t,\bar Z_t,\bar u_t,p_t,q_t,k_t)
}{
u_t-\bar u_t
}
\right]\ge0.
\end{equation}
Moreover, for any \(t\ge0\) and $a\in U$,
\begin{equation}\label{eq:smp-pointwise}
\ip{
H_u(t,\bar X_t,\bar Y_t,\bar Z_t,\bar u_t,p_t,q_t,k_t)
}{
a-\bar u_t
}
\ge0,
\qquad
\mathbb P-{\rm a.s.}
\end{equation}
Equivalently,
\[
-H_u(t,\bar X_t,\bar Y_t,\bar Z_t,\bar u_t,p_t,q_t,k_t)
\in N_U(\bar u_t),
\qquad \mathbb P-{\rm a.s.},
\]
where for $a\in U$,
\[
N_U(a)
:=
\left\{
\zeta\in\R^r:
\langle \zeta,v-a\rangle\le0
\text{ for any }v\in U
\right\}
\]
is the normal cone of the convex set \(U\).
\end{theorem}

\begin{proof}
For any \(u\in\Uad\), recall the notation
\[
v:=u-\bar u,
\qquad
u^\varepsilon:=\bar u+\varepsilon v,\qquad 0<\varepsilon\le1.
\]
Since \(U\) is convex, \(u^\varepsilon\in\Uad\). By the optimality of
\(\bar u\),
\[
0\le
\lim_{\varepsilon\downarrow0}
\frac{J(u^\varepsilon)-J(\bar u)}{\varepsilon}.
\]
Using the first-order expansion of the cost in
Lemma~\ref{lem:cost-expansion}, we obtain
\[
0\le
\sum_{t=0}^{\infty}
\E\left[
e^{-\lambda t}
\left(
\ip{l_x(t)}{\xi_t}
+\ip{l_y(t)}{\eta_t}
+\ip{l_z(t)}{\zeta_t}
+\ip{l_u(t)}{v_t}
\right)
\right].
\]
Applying the duality identity in Lemma~\ref{lem:duality-identity}, we have
\[
0\le
\sum_{t=0}^{\infty}
\E\left[
e^{-\lambda t}
\ip{
l_u(t)-b_u(t)^* p_t-\sigma_u(t)^* q_t+f_u(t)^* k_t
}{
v_t
}
\right].
\]
Then \eqref{eq:smp} follows immediately from the definition of the Hamiltonian \eqref{eq:Hamiltonian}.

We next derive the pointwise variational inequality. For a given \(t\ge0\),
\(A\in\mathcal F_t\), and a deterministic \(a\in U\), define the localized
admissible control
\[
u_s^{A,a}
=
\begin{cases}
a, & s=t,\ \omega\in A,\\
\bar u_t, & s=t,\ \omega\in A^c,\\
\bar u_s, & s\neq t.
\end{cases}
\]
Equivalently,
\[
u_t^{A,a}=\bar u_t+1_A(a-\bar u_t),
\qquad
u_s^{A,a}=\bar u_s,\quad s\neq t.
\]
Note that \(u^{A,a}\) also takes values in \(U\) and \(u^{A,a}\in\Uad\). Substituting \(u=u^{A,a}\)
into \eqref{eq:smp} yields
\[
\E\left[
1_A e^{-\lambda t}
\ip{
H_u(t,\bar X_t,\bar Y_t,\bar Z_t,\bar u_t,p_t,q_t,k_t)
}{
a-\bar u_t
}
\right]\ge0.
\]
Since \(A\in\mathcal F_t\) is arbitrary, we obtain, for this given
\(a\in U\),
\[
\ip{
H_u(t,\bar X_t,\bar Y_t,\bar Z_t,\bar u_t,p_t,q_t,k_t)
}{
a-\bar u_t
}
\ge0,
\qquad \mathbb P-{\rm a.s.}
\]
On the other hand, since \(U\subset\R^r\) is separable, choose a countable dense subset
\(D\subset U\), 
we can get a full-measure set, on which the above inequality
holds for any \(a\in D\). By continuity of the inner product with respect
to \(a\), it turns out that \eqref{eq:smp-pointwise} holds for any \(a\in U\), \(\mathbb P\)-a.s.

Finally, the normal-cone formulation is a direct result from the definition of
\(N_U\). Indeed, \eqref{eq:smp-pointwise} is equivalent to
\[
\left\langle
-H_u(t,\bar X_t,\bar Y_t,\bar Z_t,\bar u_t,p_t,q_t,k_t),
a-\bar u_t
\right\rangle
\le0,
\qquad
\mathbb P\text{-a.s.}
\]
which is actually
\[
-H_u(t,\bar X_t,\bar Y_t,\bar Z_t,\bar u_t,p_t,q_t,k_t)
\in N_U(\bar u_t),
\qquad \mathbb P\text{-a.s.}.
\]

\end{proof}

\begin{theorem}[Verification theorem]\label{thm:verification}
For $\bar u\in\Uad$,
$(\bar X,\bar Y,\bar Z,\bar N)$ is its corresponding state process, and
$(k,p,q,R)$ is the solution to the corresponding adjoint equation
\eqref{eq:adjoint}. Under Assumptions~\ref{ass:control-state}--\ref{ass:control-diff-cost},
together with that
the mapping
\[
(x,y,z,u)\mapsto
H(t,x,y,z,u,p_t,q_t,k_t)
\]
is convex,
if for any $u\in\Uad$, \(\bar u\) satisfies
\begin{equation}\label{cz5}
\sum_{t=0}^{\infty}\E\left[e^{-\lambda t}\ip{H_u(t,\bar X_t,\bar Y_t,\bar Z_t,\bar u_t,p_t,q_t,k_t)}{u_t-\bar u_t}\right]\geq0,
\end{equation}
then \(\bar u\) is an optimal control.
\end{theorem}

\begin{proof}
For any \(u\in\Uad\), denote by \((X,Y,Z,N)\) its corresponding state. Set
\[
\delta X_t=X_t-\bar X_t,
\qquad
\delta Y_t=Y_t-\bar Y_t,
\qquad
\delta Z_t=Z_t-\bar Z_t,
\qquad
\delta N_t=N_t-\bar N_t,
\qquad
\delta u_t=u_t-\bar u_t.
\]
By the convexity of \(H\) in \((x,y,z,u)\),
\begin{align}
&H(t,X_t,Y_t,Z_t,u_t,p_t,q_t,k_t)
-H(t,\bar X_t,\bar Y_t,\bar Z_t,\bar u_t,p_t,q_t,k_t)
\notag\\
&\quad\geq
\ip{H_x(t)}{\delta X_t}+\ip{H_y(t)}{\delta Y_t}
+\ip{H_z(t)}{\delta Z_t}
+\ip{H_u(t)}{\delta u_t}.
\label{eq:verification-convexity}
\end{align}
On the other hand, since \(H=l-\ip{b}{p}-\ip{\sigma}{q}+\ip{f}{k}\),
\begin{align}
l(t,X_t,Y_t,Z_t,u_t)-l(t,\bar X_t,\bar Y_t,\bar Z_t,\bar u_t)
&=H(t,X_t,Y_t,Z_t,u_t,p_t,q_t,k_t)
-H(t,\bar X_t,\bar Y_t,\bar Z_t,\bar u_t,p_t,q_t,k_t)
\notag\\
&\quad+
\ip{\delta b_t}{p_t}+\ip{\delta\sigma_t}{q_t}-\ip{\delta f_t}{k_t},
\label{eq:l-H-diff}
\end{align}
where
\begin{align*}
&\delta b_t=b(t,X_t,Y_t,Z_t,u_t)-b(t,\bar X_t,\bar Y_t,\bar Z_t,\bar u_t),\\
&\delta\sigma_t=\sigma(t,X_t,Y_t,Z_t,u_t)-\sigma(t,\bar X_t,\bar Y_t,\bar Z_t,\bar u_t),\\
&\delta f_t=f(t,X_t,Y_t,Z_t,u_t)-f(t,\bar X_t,\bar Y_t,\bar Z_t,\bar u_t).
\end{align*}
Combining \eqref{eq:verification-convexity} with \eqref{eq:l-H-diff}, we have 
\begin{align}
&\sum_{t=0}^{T}\E\left[e^{-\lambda t}
\bigl(l(t,X_t,Y_t,Z_t,u_t)-l(t,\bar X_t,\bar Y_t,\bar Z_t,\bar u_t)\bigr)\right]\geq
\sum_{t=0}^{T}\E\left[e^{-\lambda t}\ip{H_u(t)}{\delta u_t}\right]+\mathcal D_T,
\label{eq:verification-finite}
\end{align}
where
\begin{align*}
\mathcal D_T
:={}&\sum_{t=0}^{T}\E\left[e^{-\lambda t}
\left(\ip{H_x(t)}{\delta X_t}+\ip{H_y(t)}{\delta Y_t}+\ip{H_z(t)}{\delta Z_t}\right)\right]\\
&+\sum_{t=0}^{T}\E\left[e^{-\lambda t}
\left(\ip{\delta b_t}{p_t}+\ip{\delta\sigma_t}{q_t}-\ip{\delta f_t}{k_t}\right)\right].
\end{align*}

We next prove that \(\mathcal D_T\to0\). Notice
\[
\delta X_{t+1}=\delta X_t+\delta b_t+\delta\sigma_t\Delta W_t.
\]
With the help of the backward adjoint equation, we have
\begin{align}
&\E\left[e^{-\lambda(T+1)}\ip{\delta X_{T+1}}{p_{T+1}}\right]
-\E\left[\ip{\delta X_0}{p_0}\right]
\notag\\
&\quad=
\sum_{t=0}^{T}\E\left[e^{-\lambda t}
\left(\ip{\delta b_t}{p_t}+\ip{\delta\sigma_t}{q_t}\right)\right]
+
\sum_{t=1}^{T+1}\E\left[e^{-\lambda t}\ip{H_x(t)}{\delta X_t}\right].
\label{eq:verification-x-identity}
\end{align}
Since \(\delta X_0=0\), \eqref{eq:verification-x-identity} implies
\begin{align}
&\sum_{t=0}^{T}\E\left[e^{-\lambda t}
\left(\ip{\delta b_t}{p_t}+\ip{\delta\sigma_t}{q_t}+\ip{H_x(t)}{\delta X_t}\right)\right]
\notag\\
&\quad=
\E\left[e^{-\lambda(T+1)}\ip{\delta X_{T+1}}{p_{T+1}}\right]
-
\E\left[e^{-\lambda(T+1)}\ip{H_x(T+1)}{\delta X_{T+1}}\right].
\label{eq:verification-x-sum}
\end{align}
Similarly, we have
\begin{align}
&\E\left[e^{-\lambda(T+1)}\ip{\delta Y_{T+1}}{k_{T+1}}\right]
-\E\left[\ip{\delta Y_0}{k_0}\right]
\notag\\
&\quad=
\sum_{t=0}^{T}\E\left[e^{-\lambda t}
\left(\ip{H_y(t)}{\delta Y_t}+\ip{H_z(t)}{\delta Z_t}\right)\right]
-
\sum_{t=1}^{T+1}\E\left[e^{-\lambda t}\ip{\delta f_t}{k_t}\right].
\label{eq:verification-y-identity}
\end{align}
And due to \(k_0=0\), \eqref{eq:verification-y-identity} leads to
\begin{align}
&\sum_{t=0}^{T}\E\left[e^{-\lambda t}
\left(\ip{H_y(t)}{\delta Y_t}+\ip{H_z(t)}{\delta Z_t}-\ip{\delta f_t}{k_t}\right)\right]
\notag\\
&\quad=
\E\left[e^{-\lambda(T+1)}\ip{\delta Y_{T+1}}{k_{T+1}}\right]
+
\E\left[e^{-\lambda(T+1)}\ip{\delta f_{T+1}}{k_{T+1}}\right].
\label{eq:verification-y-sum}
\end{align}
Combining \eqref{eq:verification-x-sum} with \eqref{eq:verification-y-sum}, we have
\begin{align*}
\mathcal D_T={}&
\E\left[e^{-\lambda(T+1)}\ip{\delta X_{T+1}}{p_{T+1}}\right]
-\E\left[e^{-\lambda(T+1)}\ip{H_x(T+1)}{\delta X_{T+1}}\right]\\
&+\E\left[e^{-\lambda(T+1)}\ip{\delta Y_{T+1}}{k_{T+1}}\right]
+\E\left[e^{-\lambda(T+1)}\ip{\delta f_{T+1}}{k_{T+1}}\right].
\end{align*}
Actually, each term on the right-hand side of above equality converges to $0$ as $T\to\infty$. For instance, since $\delta X,p\in L^{2,\lambda}_{\F}$, as $T\to\infty$,
\[
\left|\E\left[e^{-\lambda(T+1)}\ip{\delta X_{T+1}}{p_{T+1}}\right]\right|
\leq
\left(\E[e^{-\lambda(T+1)}|\delta X_{T+1}|^2]\right)^{1/2}
\left(\E[e^{-\lambda(T+1)}|p_{T+1}|^2]\right)^{1/2}\longrightarrow0.
\]
The convergence of other terms can be gotten in the same way. Therefore, we have
$
\lim\limits_{T\to\infty}\mathcal D_T=0$.

Back to \eqref{eq:verification-finite}. By Lemma \ref{lem:cost-well-defined}, the cost functional is absolutely convergent, and consequently the left-hand side of \eqref{eq:verification-finite} converges to \(J(u)-J(\bar u)\) as $T\to\infty$. So taking \(T\to\infty\) in \eqref{eq:verification-finite} and noticing \eqref{cz5}, we have
\[
J(u)-J(\bar u)
\geq
\sum_{t=0}^{\infty}\E\left[e^{-\lambda t}\ip{H_u(t)}{u_t-\bar u_t}\right]\geq0.
\]
Hence \(J(u)\geq J(\bar u)\) for any \(u\in\Uad\), and \(\bar u\) is optimal.
\end{proof}

\subsection{Existence of optimal controls}\label{subsec:existence}

We conclude this section with a sufficient condition for the existence
of an optimal control on a prescribed admissible class
\(\mathcal K\subset\Uad\), based on the stability of the controlled FBS\(\Delta\)E and the continuity
of the cost functional in \(L^{2,\lambda}_{\F}\) space. To be specific, in this subsection we take a nonempty set \(\mathcal K\subset\Uad\) which is sequentially compact in the strong topology of \(L^{2,\lambda}_{\mathcal F}(0,\infty;\R^r)\).

\begin{proposition}[A compactness criterion for existence]\label{prop:optimal-existence}
Under Assumptions~\ref{ass:control-state}--\ref{ass:control-diff-cost}, there exists \(\bar u\in\mathcal K\) such that
\[
J(\bar u)=\inf_{u\in\mathcal K}J(u).
\]
\end{proposition}
\begin{proof}
Take \(\{u^n\}_{n\ge1}\subset\mathcal K\) which is a minimizing sequence. Sequential compactness guarantees that there exists a \(\bar{u}\in\mathcal K\) and a subsequence of \(\{u^n\}_{n\ge1}\), still denoted by \(\{u^n\}_{n\ge1}\), such that
\[
\|u^n-\bar u\|_{\lambda}\longrightarrow0\ \ \ {\rm as}\ n\to\infty.
\]
Let \((X^n,Y^n,Z^n)\) and \((\bar X,\bar Y,\bar Z)\) be the corresponding state processes of $u^n$ and $\bar{u}$, respectively.
Set
\[
\Theta_t^n:=(X_t^n,Y_t^n,Z_t^n,u_t^n)\ \ \ {\rm and}\ \ \
\bar\Theta_t:=(\bar X_t,\bar Y_t,\bar Z_t,\bar u_t).
\]
By the stability estimate in Lemma~\ref{lem:state-stability},
\[
\|\Theta^n-\bar\Theta\|_{\lambda}
\longrightarrow0.
\]

On the other hand, since each of \(l_x,l_y,l_z,l_u\) is continuous and satisfies the
linear-growth condition in
Assumption~\ref{ass:control-diff-cost}, for \(\phi\in\{l_x,l_y,l_z,l_u\}\),  we can define
\[
\bar\phi^n(t)
:=
\int_0^1
\phi\bigl(
t,\bar\Theta_t
+\rho(\Theta_t^n-\bar\Theta_t)
\bigr)\,d\rho.
\]
Lemma~\ref{lem:controlled-convergence} (ii) yields
\[
\bar l_x^n\longrightarrow
l_x(\cdot,\bar\Theta),\ \ \
\bar l_y^n\longrightarrow
l_y(\cdot,\bar\Theta),\ \ \
\bar l_z^n\longrightarrow
l_z(\cdot,\bar\Theta),\ \ \
\bar l_u^n\longrightarrow
l_u(\cdot,\bar\Theta)
\]
in their corresponding $L^{2,\lambda}_{\F}$ spaces. In particular,
\[
\sup_{n\ge1}
\left(
\|\bar l_x^n\|_{\lambda}
+\|\bar l_y^n\|_{\lambda}
+\|\bar l_z^n\|_{\lambda}
+\|\bar l_u^n\|_{\lambda}
\right)
<\infty.
\]

Then by the mean value theorem, we have
\begin{align*}
l(t,X_t^n,Y_t^n,Z_t^n,u_t^n)
-l(t,\bar X_t,\bar Y_t,\bar Z_t,\bar u_t)
=&
\left\langle
\bar l_{x}^n(t),X_t^n-\bar X_t
\right\rangle
+
\left\langle
\bar l_{y}^n(t),Y_t^n-\bar Y_t
\right\rangle
\\
&+
\left\langle
\bar l_{z}^n(t),Z_t^n-\bar Z_t
\right\rangle
+
\left\langle
\bar l_{u}^n(t),u_t^n-\bar u_t
\right\rangle .
\end{align*}
Consequently,
\begin{align*}
\sum_{t=0}^{\infty}
\E\left[
e^{-\lambda t}
\left|
l(t,X_t^n,Y_t^n,Z_t^n,u_t^n)
-l(t,\bar X_t,\bar Y_t,\bar Z_t,\bar u_t)
\right|
\right]
\le&
\|\bar l_x^n\|_{\lambda}
\|X^n-\bar X\|_{\lambda}
+
\|\bar l_y^n\|_{\lambda}
\|Y^n-\bar Y\|_{\lambda}
\\
&+
\|\bar l_z^n\|_{\lambda}
\|Z^n-\bar Z\|_{\lambda}
+
\|\bar l_u^n\|_{\lambda}
\|u^n-\bar u\|_{\lambda}.
\end{align*}
Bearing in mind that as $n\to\infty$,
\[
X^n\longrightarrow\bar X,\qquad
Y^n\longrightarrow\bar Y,\qquad
Z^n\longrightarrow\bar Z,\qquad
u^n\longrightarrow\bar u
\]
strongly in the corresponding $L^{2,\lambda}_{\F}$ spaces, we have as $n\to\infty$,
\[
\sum_{t=0}^{\infty}
\E\left[
e^{-\lambda t}
\left|
l(t,X_t^n,Y_t^n,Z_t^n,u_t^n)
-l(t,\bar X_t,\bar Y_t,\bar Z_t,\bar u_t)
\right|
\right]
\longrightarrow0,
\]
and thus
\[
J(u^n)\longrightarrow J(\bar u).
\]
Since \(\{u^n\}_{n\ge1}\) is a minimizing sequence,
\[
J(\bar u)
=
\lim_{n\to\infty}J(u^n)
=
\inf_{u\in\mathcal K}J(u),
\]
which put an end of the proof.

\end{proof}

\section{A Financial Example: Recursive Risk-Adjusted Portfolio Optimization}\label{sec:example}

We now apply the theoretical results to a discrete-time portfolio problem
with recursive risk adjustment. For simplicity, throughout this section,
\(m=n=d=1\), and the control is scalar. Consider a financial market
consisting of one risk-free asset and one risky asset. Denote by \(X_t\)
the investor's wealth, by \(u_t\) the amount invested in the risky
asset, by \(Y_t\) a recursive risk-evaluation process, and by \(Z_t\) the
sensitivity of this evaluation to the market innovation. The admissible
risky positions are constrained to the interval
\[
U=[\underline u,\overline u]\subset\R,
\qquad
\underline u\le\overline u.
\]
Consider the following model
\begin{equation*}\label{eq:portfolio-state}
\left\{
\begin{aligned}
&X_{t+1}
=X_t+r_fX_t+\theta u_t-\alpha Y_t+(\varsigma u_t+\rho Z_t)\Delta W_t,\\
&Y_t+Z_t\Delta W_t+\Delta N_t
=e^{-\lambda}\left(Y_{t+1}+aX_{t+1}+\beta Y_{t+1}+\gamma Z_{t+1}+\chi u_{t+1}\right),\\
&X_0=x_0.
\end{aligned}
\right.
\end{equation*}
Here the parameters \(r_f\), \(\theta\), and \(\varsigma\) represent the risk-free
return, the risky-asset risk premium and the risky-asset volatility,
respectively; \(\alpha\) and \(\rho\) describe the
feedback of the recursive risk variables \(Y_t\) and \(Z_t\) into the
wealth dynamics, respectively; \(a\), \(\beta\), and \(\gamma\) determine the
dependence of the recursive evaluation on the future state variables, respectively; \(\chi\) measures the direct effect of the risky
position on the recursive evaluation.

Thus in accordance with  FBS\texorpdfstring{\(\Delta\)}{Delta}E \eqref{eq:controlled-fbsde}, the coefficients in this model are
\[
b(t,x,y,z,u)=r_f x+\theta u-\alpha y,
\qquad
\sigma(t,x,y,z,u)=\varsigma u+\rho z,
\]
and
\[
f(t,x,y,z,u)=ax+\beta y+\gamma z+\chi u.
\]
Consequently, Assumption \ref{ass:control-state} is satisfied with the Lipschitz constants
\begin{align*}
&L_b^x=|r_f|,\qquad L_b^y=|\alpha|,\qquad L_b^z=0,\qquad L_b^u=|\theta|,\\
&L_\sigma^x=0,\qquad \ \ \ L_\sigma^y=0,\qquad\ \  L_\sigma^z=|\rho|,\ \ \ \  L_\sigma^u=|\varsigma|,\\
&L_f^x=|a|,\qquad \ L_f^y=|\beta|,\qquad  L_f^z=|\gamma|,\ \ \ \   L_f^u=|\chi|.
\end{align*}
Moreover, given a \(u^0\in U\), we know
\[
b(t,0,0,0,u^0)=\theta u^0,\qquad
\sigma(t,0,0,0,u^0)=\varsigma u^0,\qquad
f(t,0,0,0,u^0)=\chi u^0,
\]
and Assumption~\ref{ass:control-origin} is satisfied since
\[
\sum_{t=0}^{\infty}e^{-\lambda t}
\left(
|\theta u^0|^2+|\varsigma u^0|^2+|\chi u^0|^2
\right)<\infty.
\]

As for Assumption~\ref{ass:control-discount}, we use the same notations to have
\begin{align*}
&B_c=|\alpha|,
\qquad
\Sigma_c=|\rho|,
\qquad
\mathcal P_{0,c}=(1+|r_f|+|\beta|)^2+\gamma^2,\\
&\mathcal Q_{0,c}=\alpha^2+\rho^2,
\qquad
\mathcal K_c=(1+|r_f|+|\beta|)|\alpha|+|\gamma||\rho|,
\qquad
F_c=|a|,\\
&G_c=\sqrt{(1+|r_f|+|\beta|)^2+\gamma^2},\\
&\mathcal P_c
=(1+|r_f|+|\beta|)^2+\gamma^2
+(1+|r_f|+|\beta|)|\alpha|
+|\gamma||\rho|,\\
&\mathcal Q_c
=\alpha^2+\rho^2
+(1+|r_f|+|\beta|)|\alpha|
+|\gamma||\rho|,\\
&\mathcal R_c
=a^2
+|a|\sqrt{(1+|r_f|+|\beta|)^2+\gamma^2},\\
&\mathcal S_c
=(1+|r_f|+|\beta|)^2+\gamma^2
+|a|\sqrt{(1+|r_f|+|\beta|)^2+\gamma^2}.
\end{align*}

So we suppose the discount condition
\eqref{eq:control-global-discount},
and Assumption~\ref{ass:control-discount} is satisfied.

In this model, the running cost is in a quadratic form
\begin{equation*}\label{eq:portfolio-cost}
l(t,x,y,z,u)=\frac{q_x}{2}(x-\bar x)^2+\frac{q_y}{2}y^2+\frac{q_z}{2}z^2+\frac{r_u}{2}u^2,
\end{equation*}
where \(\bar x\) is the target, and \(q_x\geq0\), \(q_y\geq0\), \(q_z\geq0\) and \(r_u>0\) are constants and penalize wealth
deviation, recursive risk, risk sensitivity and the risky position, respectively. It is immediate that the running cost and its first-order derivatives
satisfy the growth conditions in
Assumption~\ref{ass:control-diff-cost}, with
\(\varrho_t\equiv0\).

Then according to \eqref{eq:Hamiltonian}, the Hamiltonian is
\begin{align}\label{cz6}
H(t,x,y,z,u,p,q,k)
={}&\frac{q_x}{2}(x-\bar x)^2+\frac{q_y}{2}y^2+\frac{q_z}{2}z^2+\frac{r_u}{2}u^2\notag\\
&-(r_fx+\theta u-\alpha y)p-(\varsigma u+\rho z)q+(ax+\beta y+\gamma z+\chi u)k,
\end{align}
Then the adjoint equation has the form
\begin{equation*}\label{eq:portfolio-adjoint}
\left\{
\begin{aligned}
&k_{t+1}
=k_t+q_y\bar Y_t+\alpha p_t+\beta k_t
+\bigl(q_z\bar Z_t-\rho q_t+\gamma k_t\bigr)\Delta W_t,\\
&p_t+q_t\Delta W_t+\Delta R_t
=e^{-\lambda}\left[
(1+r_f)p_{t+1}
-q_x(\bar X_{t+1}-\bar x)
-a k_{t+1}
\right],\\
&k_0=0.
\end{aligned}
\right.
\end{equation*}
and
\[
H_u(t,x,y,z,u,p,q,k)=r_u u-\theta p-\varsigma q+\chi k.
\]

Hence by Theorem \ref{thm:smp}, the maximum principle yields
\[
(r_u\bar u_t-\theta p_t-\varsigma q_t+\chi k_t)(u-\bar u_t)\geq0,
\qquad u\in U.
\]
Note that here \(U=[\underline u,\overline u]\neq\mathbb{R}\). We seek the optimal control based on the Hamiltonian \eqref{cz6}.
For the terms depending on $u$ in the Hamiltonian, we complete the square as below
\[
\frac{r_u}{2}u^2-(\theta p_t+\varsigma q_t-\chi k_t)u
=\frac{r_u}{2}(u-u_t^{\mathbb{R}})^2-
\frac{r_u}{2}|u_t^{\mathbb{R}}|^2,\ \ \ t\geq0,
\]
where $u_t^{\mathbb{R}}:=\frac{\theta p_t+\varsigma q_t-\chi k_t}{r_u}$, $t\geq0$, is the optimal control in the case that $U=\mathbb{R}$.
Since \(r_u>0\), minimizing the Hamiltonian over the closed convex set \(U=[\underline u,\overline u]\) is equivalent to minimizing the distance \(|u-u_t^{\mathbb{R}}|\). Thus the candidate optimal control in this model is
\begin{equation}\label{eq:portfolio-optimal-control}
\bar u_t=
\Pi_{[\underline u,\overline u]}\left(\frac{\theta p_t+\varsigma q_t-\chi k_t}{r_u}\right),\ \ \ t\geq0,
\end{equation}
where the projection
\[
\Pi_{[\underline u,\overline u]}(x)=
\begin{cases}
\underline u,&x<\underline u,\\
x,&\underline u\le x\le\overline u,\\
\overline u,&x>\overline u.
\end{cases}
\]
Noticing for $t\geq0$,
$
|\bar u_t|\le \max\{|\underline u|,|\overline u|\}
$, we know
$\sum_{t=0}^{\infty}\E[e^{-\lambda t}|\bar u_t|^2]<\infty
$, and thus $\bar{u}$ belongs to the admissible set.


On the other hand, since \(b,\sigma\) and \(f\) are affine and the running cost
is convex quadratic when \(q_x,q_y,q_z\ge0\) and \(r_u>0\),
the Hamiltonian is convex in \((x,y,z,u)\). Hence, by Theorem \ref{thm:verification}, any admissible control whose associated state and adjoint processes satisfy the coupled state equation, the adjoint equation, and the projection \eqref{eq:portfolio-optimal-control} is optimal. Thus, \eqref{eq:portfolio-optimal-control} provides a sufficient optimality criterion, based on the existence of solution to the corresponding state--adjoint system.

Actually, \eqref{eq:portfolio-optimal-control} gives a
projection characterization depending on the solution \((p,q,k)\) to the adjoint equation rather than a closed-form state feedback,
The terms \(\theta p_t\), \(\varsigma q_t\), and
\(-\chi k_t\) represent the contributions of wealth, martingale
sensitivity, and recursive-risk feedback, respectively, while the
projection \(\Pi_U\) enforces the portfolio constraint.

\end{document}